\documentclass[12pt,a4paper]{amsart}

\usepackage{mathtools,amssymb,bm,setspace}
\usepackage[left=26mm,right=26mm,top=24mm,bottom=24mm,foot=15mm]{geometry}
\usepackage{graphicx}
\usepackage{bbm}
\usepackage[usenames]{xcolor}
\usepackage[noadjust]{cite}

\usepackage{enumitem}
\setitemize{leftmargin=*}   % Removes margin from itemized lists (enumitem package)
\setenumerate{leftmargin=*} % Removes margin for enumerate lists (enumitem package)

\usepackage[colorlinks=true,citecolor=red,linkcolor=blue]{hyperref} % load _before_ cleveref!

\usepackage[capitalise,noabbrev]{cleveref}

\newcommand{\bea}{\begin{eqnarray}} %DR obsolete environment: not to be used (same with $$)
\newcommand{\eea}{\end{eqnarray}}

\allowdisplaybreaks

\newcommand{\Z}{\mathbb{Z}} %D \Bbb obsolete
\newcommand{\C}{\mathbb{C}}

\newcommand{\N}{\mathbb{Z}_{\ge0}}

\numberwithin{equation}{section}

\newtheorem{theorem}{Theorem}[section]
\newtheorem{lemma}[theorem]{Lemma}
\newtheorem{proposition}[theorem]{Proposition}
\newtheorem{corollary}[theorem]{Corollary}

\theoremstyle{definition}
\newtheorem{definition}[theorem]{Definition}
\newtheorem{example}[theorem]{Example}

\newtheorem{remark}[theorem]{Remark}

\newcommand{\abs}[1]{\lvert#1\rvert}

\newcommand{\pd}{\partial}
\DeclareMathOperator{\id}{id}
\newcommand{\vac}{\mathbbm{1}}
\newcommand{\parts}{\mathcal{P}} % set of all partitions
\newcommand{\suchthat}{\mspace{5mu} {:} \mspace{5mu}} % "such that" in sets
\DeclareMathOperator{\spn}{span}
\DeclareMathOperator{\im}{im}
\DeclareMathOperator{\End}{End}

\newcommand{\bilin}[2]{\left\langle #1 , #2 \right\rangle}
\newcommand{\no}[1]{\mathopen{:} #1 \mathclose{:}}         % normal ordering (prevent := or =:)
\newcommand{\lira}{\ensuremath{\lhook\joinrel\relbar\joinrel\rightarrow}} % long injection
\makeatletter
	\providecommand*{\twoheadrightarrowfill@}{%
	  \arrowfill@\relbar\relbar\twoheadrightarrow
	}
	\providecommand*{\twoheadleftarrowfill@}{%
	  \arrowfill@\twoheadleftarrow\relbar\relbar
	}
	\providecommand*{\xtwoheadrightarrow}[2][]{%
	  \ext@arrow 0579\twoheadrightarrowfill@{#1}{#2}%
	}
	\providecommand*{\xtwoheadleftarrow}[2][]{%
	  \ext@arrow 5097\twoheadleftarrowfill@{#1}{#2}%
	}
\makeatother

\newcommand{\alg}[1]{\mathfrak{#1}}                      % for Lie algebras

\newcommand{\SLA}[2]{\alg{#1}_{#2}}                      % Lie algebras like sl(2)
\newcommand{\SLSA}[3]{\alg{#1} (#2 \vert #3)}            % Lie superalgebras like gl(1|1)
\newcommand{\uaff}[2]{V^{#1}(#2)}            % universal affine VA
\newcommand{\saff}[2]{L_{#1}(#2)}            % simple affine VA
\newcommand{\uwalg}[3]{W^{#1}(#2;#3)}        % universal W-algebra
\newcommand{\swalg}[3]{W_{#1}(#2;#3)}        % simple W-algebra
\newcommand{\usl}[1]{\uaff{#1}{\SLA{sl}{3}}} % universal affine sl_3
\newcommand{\uslk}{\usl{k}}                  % universal affine sl_3
\newcommand{\centusl}{\mathcal{Z}(\usl{-3})} % centre of critical universal sl_3

\newcommand{\zhu}[1]{A(#1)}                  % zhu algebra

\newcommand{\bpsymb}{\textup{BP}}
\newcommand{\zamsymb}{\textup{Z}}

\newcommand{\ubp}[1]{\bpsymb^{#1}} % universal BP
\newcommand{\sbp}[1]{\bpsymb_{#1}} % simple BP
\newcommand{\uzam}[1]{\zamsymb^{#1}} % universal W_3
\newcommand{\szam}[1]{\zamsymb_{#1}} % simple W_3

\newcommand{\ubpk}{\ubp{k}}   % universal BP
\newcommand{\sbpk}{\sbp{k}}   % simple BP
\newcommand{\uzamk}{\uzam{k}} % universal W_3
\newcommand{\szamk}{\szam{k}} % simple W_3

\newcommand{\heis}{H}               % Heisenberg VOA
\newcommand{\lvoa}{\Pi}             % the lattice vertex algebra
\newcommand{\lmod}[2]{\Pi_{#1}(#2)} % modules
\newcommand{\rmod}[2]{R_{#1}(#2)}   % relaxed module (#1 = W3-mod, #2 = lattice VOA charge)

\newcommand{\bpsfsymb}{\sigma}          % BP spectral flow
\newcommand{\lsfsymb}{\gamma}           % Pi spectral flow

\newcommand{\tp}[1]{#1_{\textup{top}}} % top space of a module

\DeclareMathOperator{\tr}{tr}
\newcommand{\Gr}[1]{\bigl[#1\bigr]}                    % element of a Grothendieck group/ring
\newcommand{\traceover}[1]{\tr_{\raisebox{-2pt}{$\scriptstyle #1$}}} % for ch[M] = tr_M ...

\DeclareMathOperator{\chmap}{ch}
\newcommand{\ch}[1]{\chmap \Gr{#1}}                    % character
\newcommand{\fch}[2]{\ch{#1} \left(#2\right)}          % new character as function of q and ...

\newcommand{\bp}{Bershadsky--Polyakov}
\newcommand{\pbw}{Poincar\'{e}--Birkhoff--Witt}

\newcommand{\lhs}{left-hand side}

\newcommand{\rhs}{right-hand side}

\newcommand{\hw}{highest-weight}
\newcommand{\hwv}{\hw\ vector}
\newcommand{\hwvs}{\hwv s}
\newcommand{\hwm}{\hw\ module}
\newcommand{\hwms}{\hwm s}
\newcommand{\rhw}{relaxed highest-weight}
\newcommand{\rhwv}{\rhw\ vector}

\newcommand{\rhwm}{\rhw\ module}
\newcommand{\rhwms}{\rhwm s}
\newcommand{\chw}{conjugate highest-weight}
\newcommand{\chwv}{\chw\ vector}
\newcommand{\chwvs}{\chwv s}
\newcommand{\chwm}{\chw\ module}
\newcommand{\chwms}{\chwm s}
\newcommand{\vo}{vertex operator}
\newcommand{\voa}{\vo\ algebra}
\newcommand{\voas}{\voa s}
\newcommand{\svoa}{\vo\ superalgebra}

\newcommand{\vosa}{\vo\ subalgebra}

\newcommand{\ope}{operator product expansion}
\newcommand{\opes}{\ope s}
\newcommand{\qhr}{quantum hamiltonian reduction} %DR yes the _h_ is lowercase
\newcommand{\qhrs}{\qhr s}

\makeatletter
\renewcommand\author@andify{%
  \nxandlist {\unskip ,\penalty-1 \space\ignorespaces}%
    {\unskip {} \@@and~}%
    {\unskip \penalty-2 \space \@@and~}%
}
\makeatother

\begin{document}

\title[]{A realisation of the Bershadsky--Polyakov algebras and their relaxed modules}

\author[]{Dra\v zen  Adamovi\' c}
\address[Dra\v zen Adamovi\' c]{Department of Mathematics, Faculty of Science \\
University of Zagreb \\
Bijeni\v cka 30
}
\email{adamovic@math.hr}

\author[]{Kazuya Kawasetsu}
\address[Kazuya Kawasetsu]{
Priority Organization for Innovation and Excellence \\
Kumamoto University \\
Kumamoto 860-8555, Japan.
}
\email{kawasetsu@kumamoto-u.ac.jp}

\author[]{David Ridout}
\address[David Ridout]{
School of Mathematics and Statistics \\
University of Melbourne \\
Parkville, Australia, 3010.
}
\email{david.ridout@unimelb.edu.au}

\begin{abstract}
We present a realisation of the universal/simple \bp\ vertex algebras as subalgebras of the tensor product of the universal/simple Zamolodchikov vertex algebras and an isotropic lattice vertex algebra.  This generalises the realisation of the universal/simple affine vertex algebras associated to $\SLA{sl}{2}$ and $\SLSA{osp}{1}{2}$ given in \cite{AdaRea19}.  Relaxed \hwms\ are likewise constructed, conditions for their irreducibility are established, and their characters are explicitly computed, generalising the character formulae of \cite{KR-2019}.
\end{abstract}

\maketitle

%\onehalfspacing

\section{Introduction}

Let $\alg{g}$ be a finite-dimensional complex basic classical simple Lie superalgebra and let $\uaff{k}{\alg{g}}$ denote the corresponding universal affine vertex superalgebra of level $k$.  Associated to every nilpotent element $f \in \alg{g}$, or rather to every orbit of nilpotent elements, there is a vertex superalgebra $\uwalg{k}{\alg{g}}{f}$ called a (universal) W-algebra.  It is defined \cite{FeiAff92,KRW} as the cohomology of the tensor product of $\uaff{k}{\alg{g}}$ and a certain ghost \svoa.  An important problem is to understand the representation theory of $\uwalg{k}{\alg{g}}{f}$ and that of its simple quotient $\swalg{k}{\alg{g}}{f}$.

There are certain cases in which the representation theory of $\swalg{k}{\alg{g}}{f}$ is relatively well-understood.  In particular, $\uwalg{k}{\alg{g}}{0} = \uaff{k}{\alg{g}}$ and so $\swalg{k}{\alg{g}}{0}$ is the simple affine vertex superalgebra $\saff{k}{\alg{g}}$.  For admissible levels, the \hwms\ of the latter were classified in \cite{AraRat16} for $\alg{g}$ a simple Lie algebra.  On the other hand, when $f$ is principal and $k$ is admissible and nondegenerate, the representation theory of $\swalg{k}{\alg{g}}{f}$ was completely determined in \cite{AraRat15}, again for $\alg{g}$ nonsuper.  Other results in this direction may be found in \cite{ArBP,AraRat19}.

Our interest here is in the so-called \rhwms\ \cite{FeiEqu98,RidRel15} that play an important role in the representation theory of certain classes of nonrational non-$C_2$-cofinite vertex superalgebras including the universal W-algebras and many of their simple quotients.  The first classification result of this type addressed the simple \rhwms\ of $\saff{k}{\SLA{sl}{2}}$ for $k$ admissible \cite{AdaVer95}.  Recently, similar classifications for other affine vertex superalgebras have started to appear \cite{AraWei16,RidAdm17,WooAdm18,CreCos18,FutSim20,FutPos20,KawAdm20,CreAdm20}.  Moreover, \cite{KawRel19} explains how one can rigorously derive, for general affine vertex algebras, the relaxed classification from the \hw\ one.  A natural question now is how to obtain relaxed classifications for nonrational non-$C_2$-cofinite W-algebras.

One answer to this question is to use explicit singular vector formulae \cite{AK}.  However, this is limited to a very small subset of W-algebras and levels.  Another is to realise the simple W-algebras using a coset construction, if one is available.  Such constructions are generally very difficult to prove, see \cite{AraWAl19} for example, and are thus far limited to a small class of nilpotents (notably the principal ones).  Nevertheless, coset constructions provide powerful tools to analyse the representation theory of certain W-algebras.  A particularly tractable, but still important, special case concerns cosets by a Heisenberg subalgebra for which there are general tools available \cite{CreSch16}.  This includes, for example, the nonunitary minimal models of the $N=2$ superconformal algebras \cite{CreUni19}.

A somewhat more general approach is to use \qhr\ functors to construct W-algebra modules from affine vertex superalgebra modules, when the latter are well-understood.  More precisely, one can restrict the functor to category $\mathcal{O}$, try to prove that the reduction functor is surjective onto category $\mathcal{O}$ for the W-algebra, and then use the methods of \cite{KawRel19} to extend this to a classification of relaxed modules.

This approach is currently being tested in \cite{FehCla20} for the simplest nonrational non-$C_2$-cofinite W-algebras, the \bp\ algebras $\sbpk = \swalg{k}{\SLA{sl}{3}}{f_{\text{min}}}$ \cite{PolGau90,BerCon91} with $k$ nondegenerate admissible and $f_{\text{min}}$ minimal (take the lowest root vector of $\SLA{sl}{3}$ for definiteness).  This uses a detailed understanding \cite{KRW,KW,AraRep05} of the minimal reduction functor.  However, generalising this approach to other W-algebras will require a far better understanding of the corresponding reduction functors than is currently available.

Here, we study the representation theory of the \bp\ algebras using a promising alternative approach that has the benefit of constructing the \rhwms\ directly, by ``inverting'' the \qhr\ functor (see \cite{SemInv94}).  This method was pioneered in \cite{AdaRea19} for $\alg{g} = \SLA{sl}{2}$ and $\SLSA{osp}{1}{2}$.  In particular, the simple affine vertex algebra $\saff{k}{\SLA{sl}{2}}$ was, for $k \notin \N$, there realised as a subalgebra of the tensor product of a simple Virasoro vertex algebra $L^{\text{Vir}}_c$ and a lattice vertex algebra $\lvoa$ of indefinite type.  Similarly, $\saff{k}{\SLSA{osp}{1}{2}}$ was realised in the tensor product of a simple $N=1$ superconformal vertex algebra $L^{N=1}_c$, a free fermion and another lattice vertex algebra (closely related to $\lvoa$).

Moreover, the known irreducible $\saff{k}{\SLA{sl}{2}}$-modules were constructed in \cite{AdaRea19} as submodules of the tensor product of an irreducible $L^{\text{Vir}}_c$-module $M$ and an irreducible $\lvoa$-module $\lmod{r}{\lambda}$.  An especially nice observation is that the irreducible \rhwms\ were realised directly as $M \otimes \lmod{-1}{\lambda}$, where $M$ is some irreducible \hw\ $L^{\text{Vir}}_c$-module.  This neatly explains why the characters of these relaxed modules, proposed in \cite{CreMod13} and proven in \cite{KR-2019}, have the well-known irreducible Virasoro characters as factors.  Analogous realisations for irreducible $\saff{k}{\SLSA{osp}{1}{2}}$-modules at arbitrary levels also appear in \cite{AdaRea19}, again explaining the factorisations of their characters \cite{KR-2019}.

Of course, $L^{\text{Vir}}_c$ and $L^{N=1}_c$ are the principal \qhrs\ $\swalg{k}{\SLA{sl}{2}}{f_{\text{pr}}}$ and $\swalg{k}{\SLSA{osp}{1}{2}}{f_{\text{pr}}}$, respectively.  It is in this sense that tensoring with an appropriate \svoa\ inverts the reduction functor.  In this paper, we extend the results of \cite{AdaRea19} to the \bp\ algebras $\ubpk$ and $\sbpk$.  The role of the \qhrs\ will be played by the Zamolodchikov algebra $\uzamk = \uwalg{k}{\SLA{sl}{3}}{f_{\text{pr}}}$ and its simple quotient $\szamk$ \cite{ZamInf85}.  Our results may therefore be regarded as not inverting the principal (or minimal) reduction functor, but rather as inverting an (as yet undefined) affine version of the \qhr\ by stages functors introduced (for type A) in \cite{MorQua15}.

Recall the lattice vertex algebra $\lvoa$ and its irreducible modules $\lmod{-1}{\lambda}$ (see \cref{sec:lvoa} for precise definitions).  We shall prove the following results.
\begin{itemize}
	\item For all $k$, the universal \bp\ algebra $\ubpk$ is a vertex subalgebra of $\uzamk \otimes \lvoa$ (\cref{thm:realisation,thm:critrealisation}).
	\item For all $k$ such that $2k+3 \notin \N \cup \{-3\}$, the simple \bp\ algebra $\sbpk$ is a vertex subalgebra of $\szamk \otimes \lvoa$ (\cref{thm:simpleembed}).
	\item If $M$ is an irreducible \hw\ $\uzamk$-module, then $M \otimes \lmod{-1}{\lambda}$ is an indecomposable \rhw\ $\ubpk$-module that is irreducible for almost all $\lambda$ (\cref{thm:rmodirred}).
	\item If $2k+3 \notin \N \cup \{-3\}$, then the previous assertion holds with $\uzamk$ and $\ubpk$ replaced by $\szamk$ and $\sbpk$ (\cref{thm:rmodirred-simple}).
	\item Every nonordinary irreducible (conjugate) \hw\ $\ubpk$-module may be constructed as an explicitly given submodule of some $M \otimes \lmod{-1}{\lambda}$ (\cref{prop:allirredhw}).
\end{itemize}

In this paper, we will not delve deeper into the representations of the \bp\ algebras, leaving logarithmic (staggered) modules and Whittaker modules for a sequel.  We will also not prove that our relaxed construction produces all the \rhw\ $\sbpk$-modules, up to isomorphism, noting only that this follows by comparing with the classification results of \cite{FehCla20}.  Nevertheless, it would be very satisfying to prove this completeness using the framework developed here and we intend to also address this in the sequel.

Our success in generalising the results of \cite{AdaRea19} not only lends weight to the conjectural existence of affine reduction by stages functors, it also suggests a general program for elucidating the representation theory of a given (simple) W-algebra $\swalg{k}{\alg{g}}{f}$.  We shall describe this program, initially assuming that $k$ is nondegenerate admissible for simplicity.  The principal W-algebra $\swalg{k}{\alg{g}}{f_{\text{pr}}}$ is then rational \cite{AraRat15} and its representation theory is, in principle, known.  Nilpotent orbits admit a natural partial ordering via inclusions of their closures with the principal orbit being the largest and the zero orbit the smallest.  The program then amounts to iteratively inverting the affine reduction by stages functors to construct the representations of $\swalg{k}{\alg{g}}{f}$ from those of $\swalg{k}{\alg{g}}{f_{\text{pr}}}$.

For $\alg{g} = \SLA{sl}{3}$, the poset of nilpotent orbits is totally ordered, with the minimal orbit lying between the zero and principal ones.  Following the program described above means choosing a nondegenerate admissible level $k$ and using the known representation theory of the rational \voa\ $\szamk = \swalg{k}{\SLA{sl}{3}}{f_{\text{pr}}}$ to construct the representations of $\sbpk = \swalg{k}{\SLA{sl}{3}}{f_{\text{min}}}$.  This is the content of this paper.  Continuing this program, we should next attempt to use these results to construct the representation theory of $\saff{k}{\SLA{sl}{3}} = \swalg{k}{\SLA{sl}{3}}{0}$.  We intend to return to this in the future \cite{ACG}, comparing with the results of the complementary approach of \cite{AraRat16,KawRel19}.

For degenerate admissible levels, we expect that the privileged role of the principal W-algebra as the starting point of the program will be replaced by the simple exceptional W-algebras of \cite{KacRat08,ElaExc09,AraAss15}, many of which have been recently proven to be rational \cite{AraRat19}.  For example, if $\alg{g} = \SLA{sl}{3}$ and $k \in -\frac{3}{2} + \N$, then the principal reduction of $\saff{k}{\SLA{sl}{3}}$ is zero and the exceptional W-algebra is $\sbpk$.  The latter is rational \cite{ArBP} and so our program begins here.  If instead $k \in \N$, then both the principal and minimal reductions of $\saff{k}{\SLA{sl}{3}}$ are zero and the exceptional W-algebra is $\saff{k}{\SLA{sl}{3}}$ itself (which is also rational).

\bigskip

We conclude this introduction with a brief outline of the contents of the paper.  In \cref{sec:bp,sec:realisation}, we first introduce our conventions for the \bp\ algebras, as well as the Zamolodchikov algebras and the lattice vertex algebra $\lvoa$.  We then verify that we have a homomorphism from the universal \bp\ algebra to the tensor product of the universal Zamolodchikov algebra and $\lvoa$ in \cref{sec:realisation'}.  The fact that this is an embedding is proven in \cref{sec:injective}.

We next construct \rhwms\ for the \bp\ algebras, arguing in \cref{sec:almost} that these modules are ``almost-irreducible'' which means, among other things, that for almost all values of the parameters that naturally specify these modules, they are irreducible.  We also give a precise criterion for irreducibility.  This allows us to realise the simple \bp\ algebra as a submodule of such a relaxed module in \cref{sec:simple} when the level satisfies $2k+3 \notin \N$, thereby proving the simple analogue of the embedding of \cref{sec:realisation',sec:injective} for these levels.  Finally, \cref{sec:critlev} establishes critical-level analogues of these results.

\subsection*{Notation}

Given a homogeneous field $A(z)$ of conformal weight $\Delta_A$, we define operators $A_n$ and $A_{(n)}$, $n\in\Z$, by the expansions
\begin{equation} \label{eq:expansionconventions}
	A(z)=\sum_{n\in\Z-\Delta_A} A_n z^{-n-\Delta_A} = \sum_{n\in\Z} A_{(n)}z^{-n-1}.
\end{equation}

\subsection*{Acknowledgements}

We thank Thomas Creutzig and Zac Fehily for discussions relating to the results presented here.

D.A. is  partially supported   by the
QuantiXLie Centre of Excellence, a project cofinanced
by the Croatian Government and European Union
through the European Regional Development Fund - the
Competitiveness and Cohesion Operational Programme
(KK.01.1.1.01.0004).

KK's research is partially supported by
MEXT Japan ``Leading Initiative for Excellent Young Researchers (LEADER)'',
JSPS Kakenhi Grant numbers 19KK0065 and 19J01093 and
Australian Research Council Discovery Project DP160101520.

DR's research is supported by the Australian Research Council Discovery Project DP160101520 and the Australian Research Council Centre of Excellence for Mathematical and Statistical Frontiers CE140100049.

\section{\bp\ algebras} \label{sec:bp}

In this \lcnamecref{sec:bp}, we introduce the \bp\ algebras as vertex algebras.  They were originally defined independently by Polyakov \cite{PolGau90} and Bershadsky \cite{BerCon91} as nonprincipal \qhrs\ of the universal affine vertex algebras $\uslk$ associated to $\SLA{sl}{3}$.  In the framework of Kac--Roan--Wakimoto \cite{KRW}, they are simultaneously the minimal and subregular reductions.

We start by defining the universal \bp\ \voas\ $\ubpk$ in terms of generators and relations.  The \bp\ vertex algebra at the critical level $k=-3$ will be studied in \cref{sec:critlev}.
\begin{definition} \label{def:bp}
	For $k \ne -3$, the universal \bp\ \voa\ $\ubpk$ is the universal vertex algebra generated by fields $L$, $J$, $G^+$ and $G^-$ subject to the following \opes:
	\begin{equation} \label{ope:bp}
		\begin{gathered}
			J(z)J(w)\sim\frac{2k+3}{3(z-w)^2}, \qquad
			J(z)G^{\pm}(w)\sim\pm \frac{G^{\pm}(w)}{z-w}, \\
			L(z)G^+(w)\sim\frac{G^+(w)}{(z-w)^2} + \frac{\pd G^+(w)}{z-w}, \qquad
			L(z)G^-(w)\sim\frac{2G^-(w)}{(z-w)^2} + \frac{\pd G^-(w)}{z-w}, \\
			L(z)J(w)\sim -\frac{2k+3}{3(z-w)^3} + \frac{J(w)}{(z-w)^2} + \frac{\pd J(w)}{z-w}, \qquad
			G^{\pm}(z)G^{\pm}(w)\sim0, \\
			L(z)L(w)\sim \frac{c^{\bpsymb}_k}{2(z-w)^4}+\frac{2L(w)}{(z-w)^2}+\frac{\pd L(w)}{z-w}, \\
			\begin{aligned}
				G^+(z)G^-(w) &\sim \frac{(k+1)(2k+3)}{(z-w)^3}+\frac{3(k+1)J(w)}{(z-w)^2} \\
				&\quad+\frac{3\no{J(w)J(w)} + (2k+3) \pd J(w) - (k+3) L(w)}{z-w}.
			\end{aligned}
		\end{gathered}
	\end{equation}
	The central charge is
	\begin{equation}
		c^{\bpsymb}_k = -\frac{4(k+1)(2k+3)}{k+3}.
	\end{equation}
	As always, $\ubpk$ has a unique simple quotient and we shall denote it by $\sbpk$.
\end{definition}

\begin{remark}
	Both $\ubpk$ and $\sbpk$ are $\N$-graded by the eigenvalue of the zero mode $L_0$ because the conformal weights of $G^{\pm}(z)$ are $1$ and $2$, respectively.  This asymmetry is also reflected in the fact that $J(z)$ fails to be quasiprimary.  This failure may be rectified by instead choosing the conformal vector to be
\begin{equation}
	\widetilde{L} = L - \frac{1}{2} \pd J.
\end{equation}
Both $G^+$ and $G^-$ will then have conformal weight $\frac{3}{2}$.  The only downside is that $\ubpk$ and $\sbpk$ are now $\frac{1}{2} \N$-graded by the eigenvalue of $\widetilde{L}_0$.  Unless otherwise indicated, we shall keep $L$ as the conformal vector.
\end{remark}

The commutation relations of the modes are easily computed from the \opes\ \eqref{ope:bp}.  We record them for convenience.
\begin{equation} \label{cr:bp}
	\begin{gathered}
		[J_m,J_n] = \frac{2k+3}{3}m \delta_{m+n,0}, \qquad
		[J_m,G^\pm_n] = \pm G^\pm_{m+n}, \\
    [L_m,G^+_n] = -n G^+_{m+n}, \qquad
    [L_m,G^-_n] = (m-n) G^-_{m+n}, \\
    [L_m,J_n] = -nJ_{m+n} - \frac{2k+3}{3} \frac{m^2-m}{2} \delta_{m+n,0}, \qquad
		[G_m^\pm,G_n^\pm] = 0, \\
		[L_m,L_n] = (m-n)L_{m+n} - \frac{(2k+3)(k+1)}{k+3} \frac{m^3-m}{3} \delta_{m+n,0}, \\
    \begin{aligned}
			[G_m^+,G_n^-] &= 3 \no{JJ}_{m+n} - (k+3)L_{m+n} + \bigl( km-(2k+3)(n+1) \bigr) J_{m+n} \\
	    &\quad+ (k+1)(2k+3) \frac{m^2-m}{2} \delta_{m+n,0}.
		\end{aligned}
	\end{gathered}
\end{equation}

The associative algebra of modes specified by these relations admits a useful family of automorphisms called \emph{spectral flow automorphisms}.  These may be lifted to maps on $\ubpk$, following \cite{LiPhy97}, by introducing
\begin{equation} \label{eq:Li}
	\Lambda(\ell J,z) = z^{-\ell J_0} \prod_{n=1}^{\infty} \exp \left( \frac{(-1)^n}{n} \ell J_n z^{-n} \right)
\end{equation}
and defining the result of acting with the spectral flow map $\bpsfsymb^{\ell}$, $\ell \in \Z$, on a field $A(z)$ of $\ubpk$ to be
\begin{equation}
	\bpsfsymb^{\ell}(A(z)) = Y \bigl( \Lambda(\ell J,z) A,z \bigr),
\end{equation}
where $Y$ is the vertex map of $\ubpk$.  In particular, we have
\begin{equation} \label{eq:bpsf}
	\begin{gathered}
		\bpsfsymb^{\ell}(G^{\pm}(z)) = z^{\mp \ell} G^{\pm}(z), \qquad
		\bpsfsymb^{\ell}(J(z)) = J(z) - \frac{2k+3}{3} \ell z^{-1}, \\
		\bpsfsymb^{\ell}(L(z)) = L(z) - \ell z^{-1} J(z) + \frac{2k+3}{3} \frac{\ell(\ell+1)}{2} z^{-2}.
	\end{gathered}
\end{equation}
One can check explicitly that spectral flow preserves the defining \opes\ \eqref{ope:bp} (and the vacuum) of $\ubpk$.  These spectral flows may, moreover, be extended to $\ell \in \frac{1}{2} \Z$ if we allow half-integer modes for $G^{\pm}(z)$ as when acting on twisted $\ubpk$-modules.

Let $M$ be a $\ubpk$-module.  By twisting the action of $\ubpk$ on $M$ by the spectral flow map $\bpsfsymb^{-\ell}$, we may give $M$ a new structure as a $\ubpk$-module.  We shall denote this new $\ubpk$-module by $\bpsfsymb^{\ell}(M)$.  Denoting its elements by $\bpsfsymb^{\ell}(v)$, where $v \in M$, the twisted action is explicitly realised as
\begin{equation} \label{eq:sfmod}
	A \cdot \bpsfsymb^{\ell}(v) = \bpsfsymb^{\ell} \left( \bpsfsymb^{-\ell}(A) v \right).
\end{equation}
In this way, spectral flow lifts to invertible endofunctors on the category of $\ubpk$-modules.  Consequently, $\bpsfsymb^{\ell}(M)$ is irreducible if and only if $M$ is.

\section{Realisation of $\ubpk$} \label{sec:realisation}

Here, we realise the universal \bp\ algebra $\ubpk$ as a vertex subalgebra of the tensor product of the principal \qhr\ of $\uslk$, which we shall refer to as the Zamolodchikov algebra, and an isotropic lattice vertex algebra $\lvoa$.  We first introduce these vertex algebras and some of their modules.

\subsection{The Zamolodchikov algebras $\uzamk$ and $\szamk$} \label{W23}

As with the universal \bp\ algebras, the universal Zamolodchikov algebras \cite{ZamInf85} may also be defined in terms of generators and relations.
\begin{definition}
	For $k \ne -3$, the universal Zamolodchikov \voa\ $\uzamk$ is the universal vertex algebra generated by fields $T$ and $W$ subject to the following \opes:
	\begin{equation} \label{ope:z}
		\begin{gathered}
			T(z)T(w)\sim \frac{c^{\zamsymb}_k}{2(z-w)^4} + \frac{2T(w)}{(z-w)^2} + \frac{\pd T(w)}{z-w}, \\
			T(z)W(w)\sim \frac{3W(w)}{(z-w)^2} + \frac{\pd W(w)}{z-w}, \\
			\begin{aligned}
				W(z)W(w) &\sim \frac{(k+3)^3}{3} \left[ \frac{2\Lambda(w)}{(z-w)^2} + \frac{\pd \Lambda(w)}{z-w} \right] \\
				&\mspace{-50mu} + A \left[ \frac{c^{\zamsymb}_k}{3(z-w)^6} + \frac{2T(w)}{(z-w)^4} + \frac{\pd T(w)}{(z-w)^3} + \frac{\frac{3}{10} \pd^2 T(w)}{(z-w)^2} + \frac{\frac{1}{15} \pd^3 T(w)}{z-w} \right].
			\end{aligned}
		\end{gathered}
	\end{equation}
	Here, the central charge is
	\begin{equation}
		c^{\zamsymb}_k = -\frac{2(3k+5)(4k+9)}{k+3}
	\end{equation}
	and we have defined
	\begin{equation}
		A = -\frac{(k+3)^2(3k+4)(5k+12)}{6} \quad \text{and} \quad
		\Lambda = \no{TT} - \frac{3}{10} \pd^2 T,
	\end{equation}
	for convenience.  The unique simple quotient of $\uzamk$ will be denoted by $\szamk$.
\end{definition}

\begin{remark}
	The definition of $\uzamk$ given above differs from the standard one in that we have renormalised the field $W(z)$ by a factor of $\sqrt{A}$.  This removes a singularity in the standard definition when $c^{\zamsymb}_k = -\frac{22}{5}$, hence $k=-\frac{4}{3}$ or $-\frac{12}{5}$.  At this central charge, $W$ and $\Lambda$ belong to the maximal ideal of $\uzam{-4/3} = \uzam{-12/5}$.  The simple quotient is in fact the Virasoro minimal model $M(2,5)$ (also known as the Yang--Lee model).
\end{remark}

\subsection{The vertex algebra $\lvoa$ and its modules} \label{sec:lvoa}

Consider the abelian Lie algebra $\alg{h} = \spn_{\C}\{a,b\}$, equipped with the bilinear form defined by
\begin{equation}
	\bilin{a}{a} = - \bilin{b}{b} = 1 \quad \text{and} \quad \bilin{a}{b} = 0.
\end{equation}
For convenience, we let
\begin{equation}
	c=a-b \quad \text{and} \quad d=a+b,
\end{equation}
noting that these elements of $\alg{h}$ are isotropic: $\bilin{c}{c} = \bilin{d}{d}  = 0$.  The group algebra $\C[\Z c] = \spn_{\C} \{e^{nc} \suchthat n \in \Z\}$ then becomes an $\alg{h}$-module with action
\begin{equation}
	h e^{nc} = \bilin{h}{c} e^{nc}, \quad h \in \alg{h}.
\end{equation}
Let $\heis$ denote the Heisenberg vertex algebra associated to $\alg{h}$.
\begin{definition}
	Let $\lvoa$ denote the lattice vertex algebra $\heis \otimes \C[\Z c]$, where the action of $h \in \alg{h}$ on $\C[\Z c]$ is identified with the action of the zero mode $h_0$ of $h(z) \in \heis$.  We equip $\lvoa$ with the conformal structure given by
	\begin{equation}
		t(z) = \frac{1}{2} \no{c(z) d(z)} + \frac{2k+3}{3} \pd c(z) - \frac{1}{2} \pd d(z),
	\end{equation}
	so that $a$ and $b$ both have conformal weight $1$, whilst the weight of $e^{nc}$, $n \in \Z$, is $n$.  The central charge is
	\begin{equation}
		c^{\lvoa}_k = 2 + 8(2k+3).
	\end{equation}
\end{definition}

Vertex algebras like $\lvoa$ were studied in \cite{BerRep02}, under the name ``half-lattice vertex algebras'', as were their representation theories.  We refer to \cite[Sec.~4]{AdaRea19} for a convenient summary.  Here, we record the following \ope\ for future convenience:
\begin{equation} \label{eq:lvoaope}
	e^c(z) e^{-c}(w) = \vac + c(w) (z-w) + \frac{1}{2} \bigl( \no{c(w)c(w)} + \pd c(w) \bigr) (z-w)^2 + \cdots.
\end{equation}

Before introducing the $\lvoa$-modules relevant to what follows, we discuss spectral flow for $\lvoa$.  In analogy with \eqref{eq:Li}, we set \cite{LiPhy97}
\begin{equation} \label{eq:defj}
	\Lambda(\ell j,z) = z^{-\ell j_0} \prod_{n=1}^{\infty} \exp \left( \frac{(-1)^n}{n} \ell j_n z^{-n} \right), \quad j = b + \frac{k+3}{3} c,
\end{equation}
and define the action of the spectral flow map $\lsfsymb^{\ell}$, $\ell \in \Z$, on a field $A(z)$ of $\lvoa$ by
\begin{equation}
	\lsfsymb^{\ell}(A(z)) = Y(\Lambda(\ell j,z) A,z),
\end{equation}
where $Y$ is now the vertex map of $\lvoa$.  The reason for taking the particular element $j \in \alg{h}$ will become clear in \cref{sec:realisation'}.

With this, it is easy to verify that spectral flow acts on the generators of $\lvoa$ as follows:
\begin{equation} \label{eq:lvoasf}
	\begin{aligned}
		\lsfsymb^{\ell}(a(z)) &= a(z) - \tfrac{k+3}{3} \ell   z^{-1}, \\
		\lsfsymb^{\ell}(b(z)) &= b(z) - \tfrac{k}{3} \ell z^{-1},
	\end{aligned}
	\quad	\lsfsymb^{\ell}(e^{nc}(z)) = z^{-\ell n} e^{nc}(z) \quad \text{($n \in \Z$)}.
\end{equation}
The $\lsfsymb^{\ell}$ thus preserve \opes\ (and the vacuum) of $\lvoa$, but do not preserve its conformal structure:
\begin{equation}
	\lsfsymb^{\ell}(t(z)) = t(z) - \frac{1}{2} \ell z^{-1} j(z) + \frac{2k+3}{3} \frac{\ell(\ell+1)}{2} z^{-2}.
\end{equation}

As in \cref{sec:bp}, twisting by spectral flow defines invertible functors on the category of $\lvoa$-modules.  This therefore associates to every $\ell \in \Z$ and every (irreducible) $\lvoa$-module $M$ a new (irreducible) $\lvoa$-module $\lsfsymb^{\ell}(M)$.  Again, this generalises to $\ell \in \frac{1}{2} \Z$ and twisted modules.

The $\lvoa$-modules of interest here may be obtained by considering the $\C[\Z c]$-module generated by $e^h \in \C[\alg{h}]$ and inducing.  For later convenience, we shall generally write $h$ as a linear combination of the basis vectors $j$, defined in \eqref{eq:defj}, and $c$.  It will also be convenient to introduce
\begin{equation} \label{eq:defi}
	i = d-j = a - \frac{k+3}{3} c.
\end{equation}
We then have $\bilin{j}{j} = \frac{2k+3}{3} = -\bilin{i}{i}$, $\bilin{j}{c} = 1 = \bilin{i}{c}$ and $\bilin{i}{j} = 0$.

For each $r \in \frac{1}{2} \Z$ and $\lambda \in \C$, define
\begin{equation}
	\lmod{r}{\lambda} = \lvoa \cdot e^{rj+\lambda c}.
\end{equation}
Recall that an indecomposable module is positive-energy, with respect to a given conformal structure, if its conformal weights are bounded below.  Its top space is then the eigenspace corresponding to the minimal conformal weight (should it exist).

The irreducible weight $\lvoa$--modules were constructed and classified in \cite{BerRep02}. In particular, this work implies that $\lmod{r}{\lambda}$ is an irreducible $\lvoa$-module for $r \in \Z$. One can modify this result and show that $\lmod{r}{\lambda}$ is an irreducible twisted $\lvoa$-module for $r \in \Z+\frac{1}{2}$.  Thus we have (see also \textup{\cite[Prop.~4.1]{AdaRea19}}):
\begin{proposition} \label{prop:latmods}
	\leavevmode
 	\begin{itemize}
		\item For $r \in \Z$ and $\lambda \in \C$, $\lmod{r}{\lambda}$ is an irreducible (untwisted) $\Z$-graded $\lvoa$-module.
		\item For $r \in \Z+\frac{1}{2}$ and $\lambda \in \C$, $\lmod{r}{\lambda}$ is an irreducible $\frac{1}{2} \Z$-graded ($e^{\pi \sqrt{-1} i_0}$-twisted) $\lvoa$-module.
		\item In both cases, we have $\lmod{r}{\lambda} \cong \lmod{r}{\lambda+n}$ for all $n \in \Z$.  Otherwise, the $\lmod{r}{\lambda}$ are mutually inequivalent.
		\item $c_0$ acts on $\lmod{r}{\lambda}$ as $r$ times the identity and the $e^{nc}_{-n(r+1)}$, $n \in \Z$, act injectively.
		\item $\lmod{r}{\lambda}$ is positive-energy if and only if $r=-1$.  $\lmod{-1}{\lambda}$ is thus a \rhw\ $\lvoa$-module and its top space $\tp{\lmod{-1}{\lambda}}$ has conformal weight $\frac{2k+3}{3}$.
		\item Twisting the action of $\lvoa$ by the spectral flow maps $\lsfsymb^{\ell}$ gives (see \textup{\cite[Prop.~4.1]{AdP}} for a similar calculation)
		\begin{equation} \label{eq:sf}
			\lsfsymb^{\ell}(\lmod{r}{\lambda}) \cong \lmod{r+\ell}{\lambda}.
		\end{equation}
	\end{itemize}
\end{proposition}

Here, a \emph{\rhw\ $\Pi$-module} is a module that is generated by a single \rhwv, the latter being a weight vector that is annihilated by the $h_n$, with $h \in \alg{h}$ and $n>0$, and the $e^{mc}_n$ with $m \in \Z$ and $n>0$.  Note that the parameter $r$ in $\lmod{r}{\lambda}$ is a spectral flow index while $\lambda$ represents the eigenvalue of $i_0$.  The isomorphism class of $\lmod{r}{\lambda}$ thus only depends on $r$ and the image of $\lambda$ in $\C/\Z$.  Obviously, the vacuum module is $\lmod{0}{0}$.

It is also straightforward to determine the characters of the \rhwms\ $\lmod{-1}{\lambda}$.
\begin{proposition} \label{prop:latchars}
	Let $\delta(z) = \sum_{n \in \Z} z^n$ as usual.  Then, the character of $\lmod{-1}{\lambda}$ is
	\begin{equation}
		\fch{\lmod{-1}{\lambda}}{y,z;q} = \traceover{\lmod{-1}{\lambda}} y^{c_0} z^{i_0} q^{t_0 - c^{\lvoa}/24} = \frac{y^{-1} z^{\lambda}}{\eta(q)^2} \delta(z).
	\end{equation}
\end{proposition}

\subsection{Realisation} \label{sec:realisation'}

The (easily verified) central charge relation $c^{\bpsymb}_k = c^{\zamsymb}_k + c^{\lvoa}_k$ suggests that the three \voas\ $\ubpk$, $\uzamk$ and $\lvoa$ might be related.  The following result determines this relation precisely.
\begin{theorem} \label{thm:realisation}
	For $k \ne -3$, there is an \emph{injective} \voa\ homomorphism $\phi^k \colon \ubpk \to \uzamk \otimes \lvoa$, uniquely determined by
	\begin{equation} \label{eq:realisation}
		\begin{gathered}
			G^+ \mapsto \vac \otimes e^c, \qquad J \mapsto \vac \otimes j, \qquad L \mapsto T \otimes \vac + \vac \otimes t, \\
			\begin{aligned}
				G^- &\mapsto \left(W + \tfrac{1}{2} (k+2)(k+3) \pd T\right) \otimes e^{-c} + (k+3) T \otimes i_{-1} e^{-c} \\
				&\mspace{60mu} - \vac \otimes \left(i_{-1}^3 + 3(k+2) i_{-2} i_{-1} + 2(k+2)^2 i_{-3}\right) e^{-c}.
			\end{aligned}
		\end{gathered}
	\end{equation}
	Here, $i$ and $j$ were defined in \eqref{eq:defi} and \eqref{eq:defj}, respectively.
\end{theorem}
\begin{proof}[Sketch of proof that $\phi^k$ is a \voa\ homomorphism]
	Because $\ubpk$ is universal, it suffices to show that the \opes\ \eqref{ope:bp} of the generators $J$, $L$ and $G^{\pm}$ match those of their $\phi^k$-images.  This can be checked explicitly from the defining \opes\ \eqref{ope:z} of $\uzamk$ and those, for example \eqref{eq:lvoaope}, of $\lvoa$.  We used the computer algebra package \textsc{OPEdefs} \cite{ThiOPE91} for this purpose, but the computations are also easily performed by hand.

	For example, the coefficient of the third-order pole of $G^+(z) G^-(w)$ in \eqref{ope:bp} is determined by $G^+_2 G^- = (k+1)(2k+3) \vac$ whilst the corresponding calculation for their $\phi^k$-images proceeds as follows.  First, note that
	\begin{equation}
		\phi^k(G^+)_2 \phi^k(G^-) = -\vac \otimes \left[e^c_2, i_{-1}^3 + 3(k+2) i_{-2} i_{-1} + 2(k+2)^2 i_{-3}\right] e^{-c},
	\end{equation}
	because $e^c_2$ annihilates both $e^{-c}$ and $i_{-1} e^{-c}$.  Since $[e^c_m, i_n] = -e^c_{m+n}$ and $e^c_{-1} e^{-c} = \vac$, this indeed evaluates to
	\begin{align}
		\phi^k(G^+)_2 \phi^k(G^-)
		&= \vac \otimes \left(e^c_1 i_{-1}^2 + 3(k+2) e^c_0 i_{-1} + 2(k+2)^2 e^c_{-1}\right) e^{-c} \\
		&= \vac \otimes \left(-e^c_0 i_{-1} - 3(k+2) e^c_{-1} + 2(k+2)^2 e^c_{-1}\right) e^{-c} \notag \\
		&= \vac \otimes \left(1 - 3(k+2) + 2(k+2)^2 \right) e^c_{-1} e^{-c} \notag \\
		&= (k+1) (2k+3) \vac. \notag
	\end{align}
	Similar computations determine that all the singular coefficients match, hence that $\phi^k$ is a homomorphism.
\end{proof}

\begin{remark} \label{rem:sfmatch}
	Comparing the spectral flow maps $\bpsfsymb^{\ell}$ of $\ubpk$ and $\lsfsymb^{\ell}$ of $\lvoa$, we see that the explicit realisation \eqref{eq:realisation} requires $\bpsfsymb^{\ell} = \id \otimes \lsfsymb^{\ell}$.  In the definition \eqref{eq:defj} of the spectral flow maps of $\lvoa$, we could have replaced $j$ by any $h \in \alg{h}$ and still preserved the \opes.  However, the above realisation singles out $h=j$ as being particularly useful for our purposes.
\end{remark}

\begin{remark}
	The universal Zamolodchikov algebras $\uzamk$ admit free field realisations inside a rank-$2$ Heisenberg \voa\ $H'$, equipped with a positive-definite bilinear form \cite{ZamDis86}.  Tensoring $H'$ with $\lvoa$ thus gives free field realisations of the universal \bp\ algebras $\ubpk$.  This is not the standard free field realisation \cite{BerCon91}, but is obtained from it by bosonising the $\beta\gamma$ ghost system.  Free field realisations of the simple \bp\ algebras $\sbpk$ are less common.  However, such realisations are known for $k=-\frac{9}{4}$ \cite{CreCos13} and $k=-\frac{5}{3}$ \cite{AK}.  At other levels, we expect that BRST-type realisations exist instead.
\end{remark}

We will prove that $\phi^k$ is injective in \cref{sec:injective}.  Granting this, it follows from \cref{thm:realisation} that for $k\ne-3$, any $\uzamk \otimes \lvoa$-module is a $\ubpk$-module, by restriction.  Combining this with \cref{prop:latmods}, we obtain a construction of many positive-energy $\ubpk$-modules.  \cref{prop:latchars} then gives their characters.
\begin{corollary} \label{cor:ubp-rhwms}
	Let $k\ne-3$ and suppose that $M$ is a $\uzamk$-module that admits a character: $\fch{M}{q} = \traceover{M} q^{T_0 - c^{\zamsymb}_k  /24}$.  Then, the $\ubpk$-module $\rmod{M}{\lambda} = M \otimes \lmod{-1}{\lambda}$ is \rhw\ with character
	\begin{align}
		\fch{\rmod{M}{\lambda}}{z;q}
		&= \traceover{M \otimes \lmod{-1}{\lambda}} z^{J_0} q^{L_0-c^{\bpsymb}_k/24} \\
		&= \fch{M}{q} \fch{\lmod{-1}{\lambda}}{z^{(2k+3)/3},z;q} \notag \\
		&= \frac{z^{\lambda - (2k+3)/3} \fch{M}{q}}{\eta(q)^2} \delta(z). \notag
	\end{align}
\end{corollary}

Recall that a \emph{\rhw} $\ubpk$- or \emph{$\sbpk$-module} is a module generated by a single \rhwv, the latter being a weight vector that is annihilated by the $L_n$, $J_n$ $G^+_n$ and $G^-_n$ with $n>0$.  When $M$ is irreducible, we shall show in \cref{sec:almost} that the $\rmod{M}{\lambda}$ are always indecomposable and, in fact, are irreducible for almost all $\lambda$.

\section{Injectivity of $\phi^k$} \label{sec:injective}

In this \lcnamecref{sec:injective}, we show that the homomorphism $\phi^k \colon \ubpk \to \uzamk \otimes \lvoa$ is injective, for $k\ne-3$, and thereby prove \cref{thm:realisation}.  Recall that a partition is a finite sequence of positive integers $\mu = (\mu_1, \mu_2, \dots, \mu_{\ell})$ of length $\ell = \ell(\mu) \in \N$ satisfying
\begin{equation}
	\mu_1 \ge \mu _2 \ge \dots \ge \mu_{\ell}.
\end{equation}
The weight of the partition $\mu$ is defined to be $\abs{\mu} = \mu_1 + \mu_2 + \dots + \mu_{\ell}$.  Let $\parts$ denote the set of all partitions.

Given a partition $\mu \in \parts$ of length $\ell$ and an element $A$ of a vertex algebra, we introduce (whenever it makes sense) the convenient notation
\begin{equation} \label{eq:modepartitions}
	\begin{aligned}
		A_{+\mu} &= A_{\mu_{\ell}} \cdots A_{\mu_2} A_{\mu_1}, &
		A_{-\mu} &= A_{-\mu_1} A_{-\mu_2} \cdots A_{-\mu_{\ell}}, \\
		A_{(+\mu)} &= A_{(\mu_{\ell})} \cdots A_{(\mu_2)} A_{(\mu_1)}, &
		A_{(-\mu)} &= A_{(-\mu_1)} A_{(-\mu_2)} \cdots A_{(-\mu_{\ell})},
	\end{aligned}
\end{equation}
recalling the conventions for mode indices in \eqref{eq:expansionconventions}.  We shall also write $\mu+n$ (and $-\mu-n$) in the above to indicate that every part of $\mu \in \parts$ should be increased by $n \in \N$.  The following \lcnamecref{lem:pbwbpzam} is now clear from universality (see \cite[Thm.~4.1(b)]{KW}).
\begin{lemma}\label{lem:pbwbpzam}
	\leavevmode
	\begin{itemize}
		\item The universal \bp\ algebra $\ubpk$ has a \pbw-type basis $B_{\bpsymb} = \{J_{(-\mu)} G^+_{(-\nu)} L_{(-\rho)} G^-_{(-\sigma)} \vac \suchthat \mu, \nu, \rho, \sigma \in \parts\}$.
		\item The universal Zamolodchikov algebra $\uzamk$ likewise has a \pbw-type basis $B_{\zamsymb} = \{T_{(-\mu)} W_{(-\nu)} \vac \suchthat \mu, \nu \in \parts\}$.
	\end{itemize}
\end{lemma}

Because the lattice \voa\ $\lvoa$ restricts, as an $\heis$-module, to an infinite direct sum of Fock modules (one for each $e^{nc} \in \lvoa$), we get our third basis.
\begin{lemma} \label{lem:pbwlvoa}
	The lattice \voa\ $\lvoa$ has a \pbw-type basis $B_{\lvoa}' = \{j_{(-\mu)} c_{(-\nu)} e^{nc} \suchthat \mu, \nu \in \parts\ \text{and}\ n \in \Z\}$.
\end{lemma}
\noindent This basis will, however, need some finessing.  Let $S_m(c)$, $m\in\Z_{\ge0}$, denote the Schur function in the (commuting) variables $c_{(-n)}$, $n \in \Z_{\ge1}$, corresponding to the partition $(m)$.  Equivalently, these functions may be defined by the following special case of the Cauchy identity:
\begin{equation} \label{eq:Schur2}
	\prod_{n=1}^{\infty} \exp\left(\frac{c_{(-n)}}{n} z^n\right) = \sum_{m=0}^{\infty} S_m(c) z^m.
\end{equation}
In particular, we have
\begin{equation}
	S_0(c) = 1, \quad S_1(c) = c_{(-1)} \quad \text{and} \quad S_2(c) = \frac{1}{2} \left(c_{(-2)} + c_{(-1)}^2\right).
\end{equation}
For general $m \in \Z_{\ge1}$, the $S_m(c)$ have the form
\begin{equation} \label{eq:Schur1}
	S_m(c) = \frac{1}{m} c_{(-m)} + [\text{terms quadratic and higher in the $c_{(-n)}$ with $n<m$}].
\end{equation}

\begin{proposition} \label{prop:betterbasis}
	The set $B_{\lvoa} = \{j_{(-\mu)} e^c_{(-\nu-1)} e^{nc} \suchthat \mu, \nu \in \parts\ \text{and}\ n \in \Z\}$ is also a basis of $\lvoa$.
\end{proposition}
\begin{proof}
	It follows easily from \eqref{eq:Schur2} and $c$ being isotropic that
	\begin{equation} \label{eq:eone}
		e^c_{(-m-1)} e^{nc} =
		\begin{cases*}
			S_m(c) e^{(n+1)c} & if $m\ge0$, \\
			0 & otherwise.
		\end{cases*}
	\end{equation}
	Setting $\ell = \ell(\nu)$, we therefore have
	\begin{align} \label{eq:triangular}
		e^c_{(-\nu-1)} e^{nc}
		&= e^c_{(-\nu_1-1)} \cdots e^c_{(-\nu_{\ell}-1)} e^{nc}
		= S_{\nu_1}(c) \cdots S_{\nu_{\ell}}(c) e^{(n+\ell)c} \\
		&= \frac{c_{(-\nu_1)} \cdots c_{(-\nu_{\ell})} e^{(n+\ell)c}}{\nu_1 \cdots \nu_{\ell}} + \cdots
		= \frac{c_{(-\nu)} e^{(n+\ell)c}}{\nu_1 \cdots \nu_{\ell}} + \cdots, \notag
	\end{align}
	where $+ \cdots$ indicates terms whose $c$-degrees are greater than $\ell$.  Composing with $j_{(-\mu)}$, it follows that the elements of $B_{\lvoa}$ are linearly independent.  Moreover, they span $\lvoa$ because an obvious inductive triangularity argument shows that \eqref{eq:triangular} may be inverted and thus solved for the basis elements of $B_{\lvoa}'$.
\end{proof}

\begin{remark}
	Note that an element of the form $j_{(-\mu)} e^c_{(-\nu')} e^{n'c} \in \lvoa$, $\mu, \nu' \in \parts$ and $n' \in \Z$, may be identified with one of the basis vectors of $B_{\lvoa}$ by isolating any parts of $\nu'$ equal to $1$.  If there are $m$ such parts, let $\nu''$ be the partition obtained from $\nu'$ by removing them.  Then, $\nu''$ has no part equal to $1$ so it may be written as $\nu+1$ for some (unique) partition $\nu$.  Setting $n=m+n'$, we get the desired form:
	\begin{equation}
		j_{(-\mu)} e^c_{(-\nu')} e^{n'c} = j_{(-\mu)} e^c_{(-\nu'')} \bigl(e^c_{(-1)}\bigr)^m e^{n'c} = j_{(-\mu)} e^c_{(-\nu-1)} e^{nc}.
	\end{equation}
\end{remark}

We now prove that $\phi^k$ is injective.
\begin{proof}[Proof of \cref{thm:realisation}]
	Recall from the explicit realisation \eqref{eq:realisation} that
	\begin{equation}
		\begin{aligned}
			\phi^k(J_{(-n)}) &= \vac \otimes j_{(-n)}, \\
			\phi^k(G^+_{(-n)}) &= \vac \otimes e^c_{(-n)}, \\
			\phi^k(L_{(-n)}) &= T_{(-n)} \otimes \vac + \text{[terms not involving $T$]} \\
			\text{and} \quad \phi^k(G^-_{(-n)}) &= \sum_{m=0}^{\infty} W_{(-n+m)} \otimes e^{-c}_{(-m-1)} + \text{[terms not involving $W$].}
		\end{aligned}
	\end{equation}
	We will show that the images of the $\ubpk$ basis vectors in $B_{\bpsymb}$ are linearly independent.  These images have the form
	\begin{align} \label{eq:theimage}
		\phi^k(J_{(-\mu)} G^+_{(-\nu)} L_{(-\rho)} G^-_{(-\sigma)} \vac)
		&= T_{(-\rho)} W_{(-\sigma)} \vac \otimes j_{(-\mu)} e^c_{(-\nu)} \bigl(e^{-c}_{(-1)}\bigr)^{\ell(\sigma)} \vac + \cdots \\
		&= T_{(-\rho)} W_{(-\sigma)} \vac \otimes j_{(-\mu)} e^c_{(-\nu)} e^{-\ell(\sigma) c} + \cdots, \notag
	\end{align}
	where $+ \cdots$ indicates a linear combination of similar terms that have either fewer $T$-modes, fewer $W$-modes, or have the same number of $T$- and $W$-modes but also have some $e^{-c}_{(-m-1)}$-modes with $m\ge1$.  In the latter case, the $W$-partition $\sigma$ is replaced by another of the same length but strictly lower weight.

	The $c \to -c$ analogues of \eqref{eq:eone} and \eqref{eq:triangular} show that the action of any $e^{-c}_{(-m-1)}$ may be expressed in terms of the action of $c$-modes, hence in terms of actions of $e^c$-modes.  We can thereby rewrite any term with an $e^{-c}_{(-m-1)}$ as a linear combination of basis terms from $B_{\zamsymb} \otimes B_{\lvoa}$.  The point is that this rewriting will not change the $j$-, $T$- and $W$-modes, in particular it will not change the fact that the corresponding $W$-partition has weight strictly lower than $\abs{\sigma}$.

	The term exhibited on the \rhs\ of \eqref{eq:theimage} is therefore the unique term, when expressed in the basis $B_{\zamsymb} \otimes B_{\lvoa}$, with $\ell(\mu)$ $j$-modes, $\ell(\rho)$ $T$-modes and $\ell(\sigma)$ $W$-modes corresponding to a partition of weight $\abs{\sigma}$.  The images on the \lhs\ are therefore linearly independent, as desired, hence $\phi^k$ is injective.
\end{proof}

\section{Almost-irreducibility} \label{sec:almost}

Our next aim is to study the irreducibility of the \rhw\ $\ubpk$-modules $\rmod{M}{\lambda}$, introduced in \cref{cor:ubp-rhwms}.  More specifically, we wish to show that $\rmod{M}{\lambda}$ is ``almost-irreducible'' (to be defined shortly) when $M$ is an irreducible \hw\ $\uzamk$-module.  Here, we first prepare the groundwork for this by proving that the $\lvoa$-modules $\lmod{-1}{\lambda}$ are almost-irreducible as modules over a certain \vosa\ $U$.

\begin{definition}\label{def:almost}
	Let $V$ be a \voa\ and $M$ an $\N$-graded $V$-module with top component $\tp{M}$.
	\begin{itemize}
		\item We say that $M$ is \emph{top-generated} if $M$ is generated by $\tp{M}$.
		\item We say that $M$ has \emph{only top-submodules} if every nonzero submodule of $M$ has a nonzero intersection with $\tp{M}$.
		\item We say that $M$ is \emph{almost-irreducible} if it is top-generated and has only top-submodules.
	\end{itemize}
\end{definition}

One motivation for introducing almost-irreducibility is to isolate a class of modules whose irreducibility is determined by its top space.  Recall that the top space $\tp{M}$ of a module over a \voa\ $V$ is naturally a module over the Zhu algebra $\zhu{V}$ \cite{ZhuMod96}.  The action of $\zhu{V}$ is of course nothing but the action of the zero modes $A_0$, $A \in V$, on the top space.
\begin{proposition} \label{prop:irred}
	If $M$ is an almost-irreducible $V$-module and $\tp{M}$ is an irreducible $\zhu{V}$-module, then $M$ is irreducible.
\end{proposition}
\begin{proof}
  Let $N$ be a nonzero submodule of $M$.  Since $M$ has only top-submodules, $N \cap \tp{M}$ is a nonzero $\zhu{V}$-module.  Since $\tp{M}$ is irreducible, we have $\tp{M} \subseteq N$.  Finally, $M$ being top-generated forces $M \subseteq N$, hence $N=M$.
\end{proof}

Another motivation is to model modules obtained by Zhu-induction \cite{ZhuMod96,LiRep94}.  More precisely, this induction functor constructs a \voa\ module from a Zhu algebra module in such a way that the top space of the former coincides with the latter.  If we now quotient the former by the sum of all submodules whose intersection with the top space is zero, then the result has the same top space but is now almost-irreducible.  In this sense, almost-irreducibility captures the notion of the ``smallest'' \voa\ module with a given top space.

\subsection{Almost-irreducibility of $\lmod{-1}{\lambda}$}

Let $U$ be the vertex subalgebra of $\lvoa$ generated by $j=b+\frac{k+3}{3}c$ and $e^c$.  Our aim is to show that the $\lmod{-1}{\lambda}$ are almost-irreducible as $U$-modules.  Recall that
\begin{equation}\label{eqn:toppi}
	\tp{\lmod{-1}{\lambda}} = \spn_{\C} \{e^{-j+(\lambda+n)c} \suchthat n \in \Z\}.
\end{equation}
We first prove top-generation.  This follows easily by determining an appropriate basis using the method of \cref{prop:betterbasis}.  An obvious basis (the analogue of that of \cref{lem:pbwlvoa}) is
\begin{equation} \label{eq:lmod-obviousbasis}
	\{j_{-\mu} c_{-\nu} e^{-j+(\lambda+n)c} \suchthat \mu, \nu \in \parts\ \text{and}\ n \in \Z\},
\end{equation}
where we recall the notation of \eqref{eq:modepartitions}.

\begin{lemma} \label{lem:lmodbasis}
	For every $\lambda \in \C$, the set
	\begin{equation}
		\{j_{-\mu} e^c_{-\nu} e^{-j+(\lambda+n)c} \suchthat \mu, \nu \in \parts\ \text{and}\ n \in \Z\}
	\end{equation}
	is a basis of $\lmod{-1}{\lambda}$.
\end{lemma}
\begin{proof}
	First, generalise \eqref{eq:eone} to
	\begin{equation}
		e^c_{-m} e^{-j+(\lambda+n)c} =
		\begin{cases*}
			S_m(c) e^{-j+(\lambda+n+1)c} & if $m\ge0$, \\
			0 & otherwise,
		\end{cases*}
	\end{equation}
	where the $S_m(c)$ are the Schur functions defined in \eqref{eq:Schur2}.  The assertion now follows using the same argument as in the proof of \cref{prop:betterbasis}.
\end{proof}
\begin{proposition} \label{prop:lmodtopgen}
	For every $\lambda \in \C$, $\lmod{-1}{\lambda}$ is top-generated as a $U$-module.
\end{proposition}

To show that $\lmod{-1}{\lambda}$ has only top-submodules, we need a preparatory \lcnamecref{lem:calcs}.  It follows easily from the commutation relations $[j_m,c_n] = m\delta_{m+n,0}$, $[e^c_m,j_n] = -e^c_{m+n}$ and $[e^c_m,c_n]=[e^c_m,e^c_n]=0$ ($m,n\in\Z$),
as well as the formula
\begin{equation} \label{eq:ec0inj}
	e^c_n e^{-j + \lambda c} = \delta_{n,0} e^{-j + (\lambda+1)c} \quad \text{($n\in\N$).}
\end{equation}

\begin{lemma} \label{lem:calcs}
	\leavevmode
	\begin{itemize}
		\item For every $\mu,\nu \in \parts$ and $\lambda \in \C$, we have
		\begin{equation}
			e^c_{+\mu} j_{-\mu} c_{-\nu} e^{-j+\lambda c} = (-1)^{\ell(\mu)} c_{-\nu} e^{-j+(\lambda+\ell(\mu))c} \ne 0.
		\end{equation}
		Moreover, if $\mu' \ne \mu$ satisfies either $\ell(\mu') > \ell(\mu)$ or $\ell(\mu') = \ell(\mu)$ and $\abs{\mu'} \ge \abs{\mu}$, then $e^c_{+\mu'} j_{-\mu} c_{-\nu} e^{-j+\lambda c} = 0$.
		\item For every $\nu \in \parts$ and $\lambda \in \C$, we have
		\begin{equation}
			j_{+\nu} c_{-\nu} e^{-j+\lambda c} = \prod_{i=1}^{\ell(\nu)} \nu_i \cdot e^{-j+\lambda c} \ne 0.
		\end{equation}
		Moreover, if $\nu' \ne \nu$ satisfies $\abs{\nu'} \ge \abs{\nu}$, then $j_{+\nu'} c_{-\nu} e^{-j+\lambda c} = 0$.
	\end{itemize}
\end{lemma}

\begin{proposition} \label{prop:lmodonly}
	For every $\lambda \in \C$, the $U$-module $\lmod{-1}{\lambda}$ has only top-submodules.
\end{proposition}
\begin{proof}
	Let $w$ be a nonzero element of $\lmod{-1}{\lambda}$.  It suffices to show that there exists a linear combination of products of modes of $U$ mapping $w$ to a nonzero element of $\tp{\lmod{-1}{\lambda}}$.  Since $i_0$ acts semisimply on $\lmod{-1}{\lambda}$, we may assume that $i_0 w = (\lambda+n) w$ for some $n \in \Z$.  Then, $w$ has the form
	\begin{equation}
		w = \sum_{\mu,\nu\in\parts} C_{\mu,\nu} j_{-\mu} c_{-\nu} e^{-j + (\lambda+n)c},
	\end{equation}
	for some $C_{\mu,\nu} \in \C$ such that $C_{\mu,\nu}=0$ for all but finitely many $(\mu,\nu) \in \parts \times \parts$.

	Since $w\ne0$, the set $S$ of partitions $\mu$ for which there exists some $\nu \in \parts$ such that $C_{\mu,\nu} \ne 0$ is nonempty.  We may therefore choose $\mu' \in S$ of maximal weight among the elements of $S$ of maximal length.  By \cref{lem:calcs}, we have a nonzero vector
	\begin{equation}
		w' = e^c_{+\mu'} w = (-1)^{\ell(\mu')} \sum_{\nu \in \parts} C_{\mu',\nu} c_{-\nu} e^{-j + (\lambda+n+\ell(\mu'))c} \ne 0.
	\end{equation}
	The set $S'$ consisting of those $\nu \in \parts$ for which $C_{\mu',\nu} \ne 0$ is clearly nonempty, hence it has an element $\nu'$ of maximal weight.  By \cref{lem:calcs} again, we have
	\begin{equation}
		j_{+\nu'} w' = (-1)^{\ell(\mu')} C_{\mu',\nu'} \prod_{i=1}^{\ell(\nu')} \nu'_i \cdot e^{-j + (\lambda+n+\ell(\mu'))c} \ne 0.
	\end{equation}
As $j_{+\nu'} e^c_{+\mu'} w$ is a nonzero element of $\tp{\lmod{-1}{\lambda}}$, the proof is complete.
\end{proof}

\begin{corollary} \label{cor:lmodalmostirr}
	$\lmod{-1}{\lambda}$ is almost-irreducible as a $U$-module.
\end{corollary}

\begin{remark}
	On the other hand, $\lmod{-1}{\lambda}$ is never irreducible as a $U$-module.  This follows from observing that the generators $j$ and $e^c$ of $U$ have $i_0$-eigenvalues $0$ and $1$, respectively, so $U$ has no elements that reduce the $i_0$-eigenvalue.  Each $e^{-j+(\lambda+n)c}$, $n \in \Z$, in $\tp{\lmod{-1}{\lambda}}$ thus generates a distinct $U$-submodule.
\end{remark}

\subsection{Almost-irreducibility of $\rmod{M}{\lambda}$} \label{sec:almostirr-rhwm}

Recall from \cref{sec:realisation'} that each $\uzamk$-module $M$ yields a $\ubpk$-module $\rmod{M}{\lambda} = M \otimes \lmod{-1}{\lambda}$.  If $M$ is irreducible and \hw, then the top space of $\rmod{M}{\lambda}$ is clearly
\begin{equation}
	\tp{\rmod{M}{\lambda}} = \tp{M} \otimes \tp{\lmod{-1}{\lambda}}.
\end{equation}
It shall prove convenient to identify $\ubpk$ with its $\phi^k$-image in $\uzamk \otimes \lvoa$, as per \cref{thm:realisation}.  Recalling the explicit form of this realisation, we introduce $\widetilde{W} \in \ubpk$ so that
\begin{equation}
	\widetilde{W} = G^+_{-1} G^- = W \otimes \vac + \alpha \pd T \otimes \vac + \beta T \otimes (i-c) - \vac \otimes \omega,
\end{equation}
where $\alpha$ %$=\tfrac{1}{2} (k+2)(k+3)$
and $\beta$ %$=k+3$
are ($k$-dependent) constants and $\omega$ %$= e^c_{-3/2} \left(i_{-1}^3 + 3(k+2) i_{-2} i_{-1} + 2(k+2)^2 i_{-3}\right) e^{-c}$
is an element of $\lvoa$.  Their precise forms will not be needed in what follows.

\begin{theorem} \label{thm:rmodonlytopsubmod}
	If $M$ is a weight $\uzamk$-module that has only top-submodules, then the relaxed \hw\ $\ubpk$-module $\rmod{M}{\lambda}$ also has only top-submodules.
\end{theorem}
\begin{proof}
	Assume that $N$ is a nonzero $\ubpk$-submodule of $\rmod{M}{\lambda}$ and choose a weight vector $w \in N$.  Since the \vosa\ $U \subset \lvoa$ is generated by $j$ and $e^c$, $\vac \otimes U$ is generated by $\vac \otimes j = J$ and $\vac \otimes e^c = G^+$.  Hence, $\vac \otimes U \subset \ubpk$.  Since $\lmod{-1}{\lambda}$ has only top-submodules as a $U$-module (\cref{prop:lmodonly}), it follows that $w$ may be sent to a nonzero element of $M \otimes \tp{\lmod{-1}{\lambda}}$ by acting with $\ubpk$.  Hence, there exists nonzero $w_0 \in N$ of the form
	\begin{equation}
		w_0 = u_0 \otimes \tp{v},
	\end{equation}
	where $u_0 \in M$ and $\tp{v} \in \tp{\lmod{-1}{\lambda}}$.

	Our aim now is to construct a nonzero element in $\tp{\rmod{M}{\lambda}}$ from $w_0$, by acting with $\ubpk$.  We do this by recursively defining weight vectors $w_n=u_n \otimes \tp{v}$, $n=1,2,\dots$, in $N$ until we achieve our aim.  Here is the definition:
	\begin{itemize}
		\item If there exists $m>0$ such that $T_m u_n \ne 0$, choose the maximal such $m$ (for definiteness only) and set
		\begin{equation} \label{eq:actwithL}
			w_{n+1} = L_m w_n = T_m u_n \otimes \tp{v} + u_n \otimes t_m \tp{v} = T_m u_n \otimes \tp{v}.
		\end{equation}
		This is then a nonzero element of $N$ and we have $u_{n+1} = T_m u_n$.
		\item If $T_m u_n = 0$ for all $m>0$, but there exists $m>0$ such that $W_m u_n \ne 0$, then choose the maximal such $m$ and set
		\begin{align} \label{eq:actwithW}
			w_{n+1}
			&= \widetilde{W}_m w_n \\
			&= W_m u_n \otimes \tp{v} - \alpha (m+1) T_m u_n \otimes \tp{v} \notag \\
			&\quad + \beta \sum_{r=0}^{\infty} T_{m+r} u_n \otimes (i_{-r} - c_{-r}) \tp{v} - u_n \otimes \omega_m \tp{v} \notag \\
			&= W_m u_n \otimes \tp{v}. \notag
		\end{align}
		This is again a nonzero element of $N$, this time with $u_{n+1} = W_m u_n$.
		\item If $T_m u_n = W_m u_n = 0$ for all $m>0$, then $u_n$ generates a \hw\ submodule of $M$.  As $M$ has only top-submodules, the intersection of this submodule with $\tp{M}$ is nonzero, hence we must have $u_n \in \tp{M}$.  Thus, $w_n \in \tp{\rmod{M}{\lambda}}$ is a nonzero element of $N$ and we are done.
	\end{itemize}
	This recursion has to terminate because the conformal weight of $u_{n+1}$ is strictly less than that of $u_n$ and $M$ is positive-energy.  There therefore exists $n$ such that $w_n \in \tp{\rmod{M}{\lambda}}$ is nonzero and so $\rmod{M}{\lambda}$ has only top-submodules.
\end{proof}

\begin{theorem} \label{thm:rmodtopgen}
	If $M$ is a top-generated weight $\uzamk$-module, then the relaxed \hw\ $\ubpk$-module $\rmod{M}{\lambda}$ is also top-generated.
\end{theorem}
\begin{proof}
	Consider the submodule of $\rmod{M}{\lambda}$ generated by $\tp{\rmod{M}{\lambda}} = \tp{M} \otimes \tp{\lmod{-1}{\lambda}}$, denoting it by $N$.  Clearly, $N$ contains the elements of the form $\tp{u} \otimes \tp{v}$, where $\tp{u} \in \tp{M}$ and $\tp{v} \in \tp{\lmod{-1}{\lambda}}$.  We shall first show that this remains true when $\tp{u}$ is replaced by any $u \in M$.  As $M$ is top-generated, it is spanned by the elements obtained from the $\tp{u}$ by acting iteratively with the $T_{-m}$ and $W_{-m}$, $m \in \Z_{>0}$.  It therefore suffices to assume that $u \otimes \tp{v} \in N$, for some $u \in M$, and then show that upon replacing $u$ by $T_{-m} u$ or $W_{-m} u$, the result is still in $N$.

	Suppose then that $u \otimes \tp{v} \in N$ for all $\tp{v} \in \tp{\lmod{-1}{\lambda}}$.  Acting with $L_{-m}$, $m \in \Z_{>0}$, then gives $T_{-m} u \otimes \tp{v} + u \otimes t_{-m} \tp{v} \in N$.  Since the $U$-module $\lmod{-1}{\lambda}$ is top-generated (\cref{prop:lmodtopgen}), $u \otimes t_{-m} \tp{v}$ may be obtained from $u \otimes \tp{v} \in N$ by acting with $U$-modes.  But $\vac \otimes U \subset \ubpk$, so we have $u \otimes t_{-m} \tp{v} \in N$ and hence $T_{-m} u \otimes \tp{v} \in N$.

	Similarly, acting with $\widetilde{W}_{-m}$, $m \in \Z_{>0}$, results in
	\begin{align}
		&W_{-m} u \otimes \tp{v} + \alpha (m-1) T_{-m} u \otimes \tp{v} \\
			&\quad + \beta \sum_{r=0}^{\infty} T_{-m+r} u \otimes (i_{-r} - c_{-r}) \tp{v} - u \otimes \omega_{-m} \tp{v} \in N. \notag
	\end{align}
	As before, $\lmod{-1}{\lambda}$ being top-generated implies that $u \otimes \omega_{-m} \tp{v} \in N$ whilst the previous argument gives $T_{-m} u \otimes \tp{v} \in N$.  An obvious hybrid of these arguments then shows that $T_{-m+r} u \otimes (i_{-r} - c_{-r}) \tp{v} \in N$ and so we conclude that $W_{-m} u \otimes \tp{v} \in N$ as well.

	It follows that $M \otimes \tp{\lmod{-1}{\lambda}} \subset N$.  One more appeal to $\lmod{-1}{\lambda}$ being top-generated then forces $M \otimes \lmod{-1}{\lambda} \subseteq N$, completing the proof.
\end{proof}

\begin{corollary} \label{cor:rhwmalmostirr}
	If $M$ is an almost-irreducible weight $\uzamk$-module, then the relaxed \hw\ $\ubpk$-module $\rmod{M}{\lambda}$ is also almost-irreducible.
\end{corollary}

\subsection{Irreducibility of $\rmod{M}{\lambda}$} \label{sec:rmodirred}

We now show that when $M$ is an irreducible \hw\ $\uzamk$-module, almost all of the \rhw\ $\ubpk$-modules $\rmod{M}{\lambda}$ that we have constructed are irreducible.  This gives an \emph{a posteriori} justification for the term ``almost-irreducible''.

We recall that a \hwv\ for $\ubpk$ is a $J_0$- and $L_0$-eigenvector that is annihilated by $G^+_0$ and all the $A_n$, $A \in \ubpk$, with $n>0$.  The definition of a \emph{\chwv} is then obtained by replacing the condition of annihilation by $G^+_0$ with annihilation by $G^-_0$.  A \chwm\ is then one which is generated by a single \chwv.

\begin{theorem} \label{thm:rmodirred}
	Let $M$ be an irreducible \hw\ $\uzamk$-module whose \hwv\ $u$ has $T_0$-eigenvalue $\Delta$ and $W_0$-eigenvalue $w$.  Then:
	\begin{itemize}
		\item The \rhw\ $\ubpk$-module $\rmod{M}{\lambda}$ is always indecomposable.
		\item $\rmod{M}{\lambda}$ is irreducible if and only if the polynomial
		\begin{equation} \label{eq:defpoly}
			p^{\Delta,w}_k(x) = w-(k+2)(k+3)\Delta + \left[(k+3)\Delta-2(k+2)^2\right] x + 3(k+2)x^2-x^3
		\end{equation}
		has no roots in the coset $\lambda + \Z$.
		\item $\rmod{M}{\lambda}$ has no \hwvs\ and its \chwvs\ are precisely the $u \otimes e^{-j+\mu c}$ with $\mu \in \lambda + \Z$ satisfying $p^{\Delta,w}_k(\mu) = 0$.
	\end{itemize}
\end{theorem}
\begin{proof}
	By \cref{prop:irred,cor:rhwmalmostirr}, $\rmod{M}{\lambda}$ is irreducible if $\tp{\rmod{M}{\lambda}}$ is an irreducible $\zhu{\ubpk}$-module.  As the latter has one-dimensional weight spaces, it is irreducible if and only if both $G^+_0$ and $G^-_0$ act bijectively.  From the one-dimensionality of the top space of $M$, the realisation \eqref{eq:realisation}, the basis \eqref{eqn:toppi} of $\tp{\lmod{-1}{\lambda}}$ and \cref{eq:ec0inj}, we see that $G^+_0$ always acts bijectively:
	\begin{equation} \label{eq:G+inj}
		G^+_0 (u \otimes e^{-j+(\lambda+n)c}) = u \otimes e^c_0 e^{-j+(\lambda+n)c} = u \otimes e^{-j+(\lambda+n+1)c} \ne 0.
	\end{equation}
	$\tp{\rmod{M}{\lambda}}$ is therefore uniserial, meaning that its submodules $S_i$ form a linear chain under inclusion: $0 \subset S_1 \subset S_2 \subset \dots \subset \tp{\rmod{M}{\lambda}}$.

	Suppose that $\rmod{M}{\lambda}$ was decomposable, hence that $\rmod{M}{\lambda} = N \oplus N'$ for nonzero submodules $N$ and $N'$.  Then, $N \cap \tp{\rmod{M}{\lambda}}$ and $N' \cap \tp{\rmod{M}{\lambda}}$ are both nonzero submodules of $\tp{\rmod{M}{\lambda}}$, because $\rmod{M}{\lambda}$ has only top-submodules (\cref{thm:rmodonlytopsubmod}).  But, their intersection is clearly zero, in contradiction to $\tp{\rmod{M}{\lambda}}$ being uniserial.  We therefore conclude that $\rmod{M}{\lambda}$ is indecomposable.

	Now, any \hwv\ of $\rmod{M}{\lambda}$ must be in $\tp{\rmod{M}{\lambda}}$ as otherwise the submodule it generates would contradict \cref{thm:rmodonlytopsubmod}.  But, \eqref{eq:G+inj} shows that $\tp{\rmod{M}{\lambda}}$ has no \hwvs.  The situation is similar for \chwvs\ except that a calculation somewhat more involved than \eqref{eq:G+inj} gives
	\begin{equation}  \label{eq:G-inj}
		G^-_0 (u \otimes e^{-j+(\lambda+n)c}) = p^{\Delta,w}_k(\lambda+n) u \otimes e^{-j+(\lambda+n-1)c}.
	\end{equation}
	Indeed, it follows by combining the following straightforward equalities:
	\begin{equation}
		\begin{aligned}
			(We^{-c})_0v
			&= e^{-c}_0W_0 u\otimes e^{-j+(\lambda+n)c}
			= wu\otimes e^{-j+(\lambda+n-1)c},\\
			[(\nu_{-1}T)e^{-c}]_0v
			&=e^{-c}_0T_0\nu_0 u\otimes e^{-j+(\lambda+n)c}
			=\Delta (\lambda+n) u\otimes e^{-j+(\lambda+n-1)c},\\
			[(DT) e^{-c}]_0v
			&=e^{-c}_0(-2T_0) u\otimes e^{-j+(\lambda+n)c}
			=-2\Delta u\otimes e^{-j+(\lambda+n-1)c},\\
			[\nu_{-1}^3e^{-c}]_0v
			&=e^{-c}_0\nu_0^3 u\otimes e^{-j+(\lambda+n)c}
			=(\lambda+n)^3 u\otimes e^{-j+(\lambda+n-1)c},\\
			[\nu_{-1}\nu_{-2}e^{-c}]_0v
			&=e^{-c}_0(-\nu_0)\nu_0 u\otimes e^{-j+(\lambda+n)c}
			=-(\lambda+n)^2 u\otimes e^{-j+(\lambda+n-1)c},\\
			[\nu_{-3}e^{-c}]_0v
			&=e^{-c}_0\nu_0 u\otimes e^{-j+(\lambda+n)c}
			=(\lambda+n) u\otimes e^{-j+(\lambda+n-1)c}.
		\end{aligned}
	\end{equation}
	Here we set $v=u \otimes e^{-j+(\lambda+n)c}$ for brevity.

	If $p^{\Delta,w}_k(\lambda+n) \ne 0$ for any $n \in \Z$, then $G^-_0$ acts bijectively and so $\rmod{M}{\lambda}$ is irreducible.  On the other hand, $p^{\Delta,w}_k(\lambda+n) = 0$ for some $n \in \Z$ implies that $u \otimes e^{-j+(\lambda+n)c}$ is a \chwv\ in $\rmod{M}{\lambda}$ generating a nonzero proper submodule.
\end{proof}

\subsection{Irreducible submodules of $\rmod{M}{\lambda}$} \label{sec:irred}

The result of the previous \lcnamecref{sec:rmodirred} explicitly realises irreducible \rhw\ $\ubpk$-modules.  We now show that we can similarly realise irreducible \hw\ $\ubpk$-modules by analysing submodules of $\rmod{M}{\lambda}$, when the latter is reducible.

\begin{proposition} \label{prop:irredhw}
	Suppose that $M$ is an irreducible \hw\ $\uzamk$-module with \hwv\ $u$ and that $\rmod{M}{\lambda}$ is reducible.  Choose $\mu \in \lambda + \Z$ with \emph{maximal} real part such that $u \otimes e^{-j+\mu c}$ is a \chwv.  Then, $u \otimes e^{-j+\mu c}$ generates an \emph{irreducible} \chw\ submodule of $\rmod{M}{\lambda}$.
\end{proposition}
\begin{proof}
	By \cref{thm:rmodirred}, $\rmod{M}{\lambda}$ being reducible implies that there exist \chwvs\ of the form $u \otimes e^{-j+\mu c}$, so we may indeed choose $\mu$ maximal.  Then, $u \otimes e^{-j+\mu c}$ is, up to nonzero multiples, the unique \chwv\ in the \chw\ submodule $C = \ubpk \cdot (u \otimes e^{-j+\mu c})$ of $\rmod{M}{\lambda}$.  Moreover, $\zhu{\ubpk} \cdot (u \otimes e^{-j+\mu c})$ is an irreducible $\zhu{\ubpk}$-submodule of $\tp{\rmod{M}{\lambda}}$.

	Choose a nonzero element $v \in C$.  Then, $D = \ubpk \cdot v \subseteq C$ is a nonzero submodule of $\rmod{M}{\lambda}$ and thus $D \cap \tp{\rmod{M}{\lambda}} \ne 0$, by \cref{thm:rmodonlytopsubmod}.  Since
	\begin{equation}
		D \cap \tp{\rmod{M}{\lambda}} \subseteq C \cap \tp{\rmod{M}{\lambda}} = \zhu{\ubpk} \cdot (u \otimes e^{-j+\mu c})
	\end{equation}
	the irreducibility of the latter implies that $u \otimes e^{-j+\mu c} \in D$, whence $D = C$.  This shows that every nonzero element of $C$ is cyclic, so this submodule is irreducible.
\end{proof}

Note that if $\mu$ is chosen maximal, as in \cref{prop:irredhw}, then the irreducible \chw\ submodule generated by $u \otimes e^{-j+\mu c}$ has an infinite-dimensional top space.

\begin{proposition} \label{prop:allirredhw}
	Every irreducible \chw\ $\ubpk$-module whose top space is infinite-dimensional may be explicitly realised as a submodule of $\rmod{M}{\lambda}$, for some irreducible \hw\ $\uzamk$-module $M$ and some $\lambda \in \C$.
\end{proposition}
\begin{proof}
	Let $C$ be an arbitrary irreducible \chw\ $\ubpk$-module with infinite-dimensional top space and let $v$ denote its \chwv.  Then, a basis of $\tp{C}$ is given by the $(G^+_0)^n v$ with $n\ge0$.

	If $u$ is the \hwv\ of some irreducible \hw\ $\uzamk$-module $M$ and $\Delta$ and $w$ are its $T_0$- and $W_0$-eigenvalues, then \cref{prop:latmods,thm:realisation} give
	\begin{equation} \label{eq:jDelta}
		\begin{aligned}
			J_0 (u \otimes e^{-j+\mu c}) &= \left( \mu - \frac{2k+3}{3} \right) u \otimes e^{-j+\mu c} \\ \text{and} \quad
			L_0 (u \otimes e^{-j+\mu c}) &= \left( \Delta + \frac{2k+3}{3} \right) u \otimes e^{-j+\mu c}.
		\end{aligned}
	\end{equation}
	We may therefore arrange for $u \otimes e^{-j+\mu c}$ to have the same $J_0$- and $L_0$-eigenvalues as $v$ by choosing $\mu$ and $\Delta$ in accordance with \eqref{eq:jDelta}.  Moreover, if we specialise $\lambda$ to $\mu$ and $w$ to the unique root of the linear (in $w$) polynomial $p^{\Delta,w}_k(\mu)$, see \eqref{eq:defpoly}, then $u \otimes e^{-j+\mu c}$ is a \chwv\ in $\rmod{M}{\lambda}$, by \cref{thm:rmodirred}.

	This shows that $u \otimes e^{-j+\mu c}$ generates a \chw\ submodule $V$ of $\rmod{M}{\lambda}$ whose irreducible quotient is isomorphic to $C$.  It moreover generates a $\zhu{\ubpk}$-module $W = V \cap \tp{\rmod{M}{\lambda}}$ with basis $u \otimes e^{-j+(\mu+n) c}$, $n\ge0$, whose irreducible quotient is isomorphic to the $\zhu{\ubpk}$-module $\tp{C}$.  Comparing bases, we see that $W \cong \tp{C}$ is irreducible and so $\mu$ must be the maximal solution in $\lambda + \Z$ of $p^{\Delta,w}_k(\mu)=0$ (otherwise, $V$ would have another \chwv\ in $V \cap \tp{\rmod{M}{\lambda}}$ contradicting irreducibility).  But now \cref{prop:irredhw} shows that $V$ is irreducible, hence $V \cong C$ as required.
\end{proof}

Of course, explicit realisations of the irreducible \chwms\ lead to explicit realisations of the irreducible \hwms\ as well, via the conjugation functor of $\ubpk$.  Similar to spectral flow functors, this arises from the automorphism of the mode algebra corresponding to
\begin{equation}
	\begin{aligned}
		G^+(z) &\mapsto G^-(z), & J(z) &\mapsto -J(z) - \frac{2k+3}{3} z^{-1}, \\
		G^-(z) &\mapsto -G^+(z), & L(z) &\mapsto L(z) - \pd J(z) - J(z) z^{-1}.
	\end{aligned}
\end{equation}
This then realises all the irreducible \hw\ $\ubpk$-modules with infinite-dimensional top spaces, but as submodules of the conjugates of the $\rmod{M}{\lambda}$.

\begin{remark}
	Suppose now that $p^{\Delta,w}_k$ has at least two roots in $\lambda + \Z$ and let $\mu$ and $\mu' < \mu$ be the maximal and next-to-maximal root, respectively.  Then, \cref{prop:irredhw} shows that $u \otimes e^{-j+\mu c}$ generates an irreducible \chwm\ $N$ of $\rmod{M}{\lambda}$ while \cref{thm:rmodirred} shows that $u \otimes e^{-j+\mu' c}$ generates a \chwm\ $N'$ that contains $N$.  Despite the fact that $\tp{N'}/\tp{N}$ is a finite-dimensional irreducible $\zhu{\ubpk}$-module, it does not necessarily follow that $N'/N$ is an irreducible $\ubpk$-module.

	The issue here is that $N'$, and hence $\rmod{M}{\lambda}$, may contain a subsingular vector which, in the language developed here, would generate a submodule that is not top-generated.  The upshot is that one can use such quotients to identify (conjugate) \hwms\ with finite-dimensional top spaces but that this does not amount to a concrete realisation.  Instead, one can employ spectral flow.
\end{remark}

\begin{corollary}
	Given an irreducible \hw\ $\ubpk$-module $N$, one of the following possibilities occurs:
	\begin{itemize}
		\item $N$ may be realised as a submodule of $\bpsfsymb^{\ell}(\rmod{M}{\lambda}) \cong M \otimes \lmod{\ell-1}{\lambda}$, for some $\ell \in \Z_{\ge1}$, some irreducible \hw\ $\uzamk$-module and some $\lambda \in \C$.
		\item The $\bpsfsymb^{-\ell}(N)$ have finite-dimensional top spaces for all $\ell \in \Z_{\ge1}$.
	\end{itemize}
\end{corollary}
\begin{proof}
	This follows from the assertion, easily checked using \eqref{eq:bpsf} and \eqref{eq:sfmod}, that the spectral flow map $\bpsfsymb^{-1}$ takes a \hwv\ to a \chwv.  For example, $G^+_n v = 0$ for $n\ge0$ implies that
	\begin{equation}
		G^+_{n+1} \bpsfsymb^{-1}(v) = \bpsfsymb^{-1} \bigl( \bpsfsymb(G^+_{n+1}) v \bigr) = \bpsfsymb^{-1} \bigl( G^+_n v \bigr) = 0.
	\end{equation}
	$\bpsfsymb^{-1}(N)$ is thus an irreducible \chwm.  If its top space is finite-dimensional, then it is also a \hwm\ and so we may apply $\bpsfsymb^{-1}$ to its \hwv.

	Iterating, we find that either the $\bpsfsymb^{-\ell}(N)$, with $\ell \in \Z_{\ge1}$, all have finite-dimensional top spaces or we arrive at an irreducible \chwm\ $\bpsfsymb^{-\ell}(N)$ with an infinite-dimensional top space.  In the latter case, $\bpsfsymb^{-\ell}(N)$ embeds into some $\rmod{M}{\lambda}$, by \cref{prop:allirredhw}.  Since spectral flow functors are invertible, we conclude that $N$ embeds into $\bpsfsymb^{\ell}(\rmod{M}{\lambda})$, as desired.
\end{proof}

It follows that we can realise any given irreducible \hw\ $\ubpk$-module, as long as its negative spectral flow orbit does not consist exclusively of modules with finite-dimensional top spaces.  A generic orbit will not have this property and therefore irreducible \hwms\ are generically realisable.  However, there are some \hwms\ that cannot be realised in this way.  In particular, the irreducible \hw\ $\sbpk$-modules with $2k+3 \in \N$ are examples because their top spaces are always finite-dimensional (see \cref{rem:examples} below).

\bigskip

We conclude with an example that illustrates these realisations.  As in \cite{Watts}, let us parametrise the highest weights of the irreducible \hw\ $\uzamk$-modules by
\begin{equation}
	\begin{aligned}
		\Delta &= \frac{(r-ts)^2 + (r-ts) (r'-ts') + (r'-ts')^2}{3t} - \frac{(t - 1)^2}{t} \\ \text{and} \quad
		w &= \frac{\bigl( r-ts - (r'-ts') \bigr) \bigl( 2(r-ts) + (r'-ts') \bigr) \bigl( r-ts + 2(r'-ts') \bigr)}{27},
	\end{aligned}
\end{equation}
where $r, r', s, s' \in {\C}$ and $t = k+3$.  By direct calculation, we see that the polynomial in \eqref{eq:defpoly} factorises.
\begin{lemma} We have
\begin{subequations}
	\begin{equation}
		p^{\Delta,w}_k(x) = -(x-x_1) (x-x_2) (x-x_3),
	\end{equation}
	where
	\begin{equation}
		\begin{aligned}
			x_1 &= t-1 - \frac{r-ts - (r'-ts')}{3}, \\
			x_2 &= t-1 + \frac{2(r-ts) + r'-ts'}{3}, \\
			x_3 &= t-1 - \frac{r-ts + 2(r'-ts')}{3}
		\end{aligned}
		\qquad \text{so} \qquad
		\begin{aligned}
			x_2 - x_1 &= r-ts, \\
			x_1 - x_3 &= r'-ts', \\
			x_2 - x_3 &= r+r'-t(s+s').
		\end{aligned}
	\end{equation}
\end{subequations}
\end{lemma}

\begin{example}
	One sees that if $k \notin \tfrac{1}{2} \Z$,  $s=s'=1$ and $r, r' \in {\Z}_{>0}$, then the roots $x_i$, $i=1,2,3$, lie in different cosets of $\C/\Z$.  Each in therefore maximal in the sense of \cref{prop:irredhw}, so this \lcnamecref{prop:irredhw} and \cref{thm:rmodirred} show that, for each $i=1,2,3$,
	\begin{equation}
		C_i = \ubpk \cdot (u \otimes e^{-j + x_i c})
	\end{equation}
	is an irreducible \chw\ $\ubpk$-module with an infinite-dimensional top space.  Applying spectral flow, it follows that $\bpsfsymb(C_i)$ is an irreducible \hwm, for each $i$.  Detailed calculation shows that $\bpsfsymb(C_1)$ and $\bpsfsymb(C_2)$ always have finite-dimensional top spaces, indeed of dimensions $r'$ and $r$ respectively.  The situation for $\bpsfsymb(C_3)$ is more subtle: if $3t$ is an integer larger than $r+r'$, then the top space is finite-dimensional (with dimension $3t-r-r'$); otherwise it is infinite-dimensional.
\end{example}

In our forthcoming publications \cite{AKR-new}, we shall present a detailed study of the structure of $\rmod{M}{\lambda}$.

\section{Realisation of the vertex algebra $\sbpk$ and its relaxed modules} \label{sec:simple}

Recall from \cref{thm:realisation} that we have established an embedding $\phi^k$ of the universal \bp\ algebra $\ubpk$, $k\ne-3$, in the tensor product of the universal Zamolodchikov algebra $\uzamk$ and the lattice vertex algebra $\lvoa$.  It is natural to ask if this realisation descends to the simple quotients, that is if $\sbpk$ embeds in $\szamk \otimes \lvoa$.  We shall show that the answer is frequently, but not always, yes.

The answer is obviously yes if $\ubpk$ is already simple.  By \cite[Thms.~0.2.1 and 9.1.2]{GorSim07}, $\ubpk$ is not simple if and only if the (noncritical) level $k$ satisfies
\begin{equation} \label{eq:fraclevel}
	k+3 = \frac{p'}{p}, \quad \text{for some coprime}\ p \in \Z_{\ge1}\ \text{and}\ p' \in \Z_{\ge2}.
\end{equation}
For these levels, we consider the projection $\pi_k \colon \uzamk \to \szamk$ and the composition
\begin{equation} \label{eq:defpsi}
	\psi^k \colon \ubpk \overset{\phi^k}{\lira} \uzamk \otimes \lvoa \xtwoheadrightarrow{\pi_k \otimes \id} \szamk \otimes \lvoa
\end{equation}
of \voa\ morphisms.  Since $\psi^k$ maps the vacuum of $\ubpk$ to the vacuum of $\szamk \otimes \lvoa$, it is not zero.  We shall investigate when $\im \psi^k$ is simple.

The following lemma about singular vectors is a version of Lemma~ 8.1 from \cite{AK}.  (We recall that a singular vector is just another name for a \hwv.)
\begin{lemma} \label{lem:easycalc}
	The vector $(G^+_{-1})^n \vac$, $n>0$, is singular in $\ubpk$ if and only if either $n=k+2$ and $k \in \{-1,0,1,2,\dots\}$ or $n=2(k+2)$ and $k \in \{-\frac{3}{2},-1,-\frac{1}{2},0,\frac{1}{2},\dots\}$.
\end{lemma}
\begin{proof}
	This is a straightforward computation using the commutation relations \eqref{cr:bp}.  Since the $J_m$, $L_m$, $G^-_{m+1}$ and $G^+_{m-1}$, with $m>0$, clearly annihilate $(G^+_{-1})^n \vac$, we only need to calculate $G^-_1 (G^+_{-1})^n \vac$.  This is easy (and has already appeared in the proof of \cite[Lem.~8.1]{AK}):
	\begin{equation}
		G^-_1 (G^+_{-1})^n \vac = -n (n-k-2) (n-2k-4) (G^+_{-1})^{n-1} \vac. \qedhere
	\end{equation}
\end{proof}

\begin{theorem} \label{thm:simpleembed}
	$\sbpk$ embeds into $\szamk \otimes \lvoa$ if and only if $2k+3 \notin \N$.
\end{theorem}
\begin{proof}
	Suppose that $\im \psi^k$ has a nonzero proper ideal $I$.  Then, $I$ is a submodule of $\szamk \otimes \lvoa = \szamk \otimes \lmod{0}{0}$.  Here, we may regard $I$ and $\szamk \otimes \lvoa$ as $\ubpk$-modules, by \eqref{eq:defpsi}, noting that this makes $\psi^k$ into a $\ubpk$-module homomorphism.

	Applying the spectral flow map $\bpsfsymb^{-1} = \id \otimes \lsfsymb^{-1}$ of $\ubpk$ ($\lsfsymb^{-1}$ is a spectral flow map of $\lvoa$, see \eqref{eq:sf}), it follows that $\bpsfsymb^{-1}(I)$ is a nonzero submodule of
	\begin{equation}
		(1 \otimes \lsfsymb^{-1})(\szamk \otimes \lmod{0}{0}) \cong \szamk \otimes \lmod{-1}{0} = \rmod{\szamk}{0}.
	\end{equation}
	Since $\szamk$ is an irreducible $\uzamk$-module, it has only top-submodules.  The $\ubpk$-module $\rmod{\szamk}{0}$ therefore also has only top-submodules, by \cref{thm:rmodonlytopsubmod}.  Thus,
	\begin{equation}
		\tp{\rmod{\szamk}{0}} \cap \bpsfsymb^{-1}(I) \ne 0.
	\end{equation}
	In other words, there exists $n \in \Z$ such that $\vac \otimes e^{-j+nc} \in \bpsfsymb^{-1}(I)$.  Applying $\bpsfsymb = \id \otimes \lsfsymb$, we conclude that $\vac \otimes e^{nc} \in I$.

	As $I \subset \im \psi^k$ and $\im \psi^k$ is a homomorphic image of $\ubpk$, $I$ has nonnegative conformal weights.  Thus, $n \in \N$.  However, $n=0$ implies that $I = \im \psi^k$, a contradiction.  Therefore, we have
	\begin{equation} \label{eq:sv}
		\psi^k \left( (G^+_{-1})^n \vac \right) = \vac \otimes e^{nc} \in I,
	\end{equation}
	for some $n \in \Z_{>0}$, using the explicit realisation of \cref{thm:realisation}.

	Now choose $n \in \Z_{>0}$ minimal such that \eqref{eq:sv} holds.  Then, $\psi^k \bigl( (G^+_{-1})^n \vac \bigr)$ is annihilated by $G^-_1$, because the result must be proportional to $\psi^k \bigl( (G^+_{-1})^{n-1} \vac \bigr)$ which is $0$ by minimality.  But, it is also annihilated by the $J_m$, $L_m$, $G^-_{m+1}$ and $G^+_{m-1}$, with $m>0$.  We conclude that $\psi^k \bigl( (G^+_{-1})^n \vac \bigr)$ is a singular vector in $I$ (regarded as a $\ubpk$-module).  As $\psi^k \bigl( (G^+_{-1})^{n-1} \vac \bigr) = \vac \otimes e^{(n-1)c}$ is nonzero, so is $(G^+_{-1})^n \vac \in \ubpk$.

	If $2k+3 \notin \N$, then this is impossible by \cref{lem:easycalc} and hence $\im \psi^k$ is simple.  On the other hand, if $k$ does have this form, then there exists $n \in \Z_{>0}$ such that $\vac \otimes e^{nc}$ is a singular vector in $\im \psi^k$ generating a proper nonzero ideal.  Thus, $\im \psi^k$ is not simple in this case.
\end{proof}

A consequence of \cref{thm:rmodirred,thm:simpleembed} is that we get families of \rhwms\ for the simple \bp\ \voa\ $\sbpk$, at least when $p \ne 1,2$.
\begin{theorem} \label{thm:rmodirred-simple}
	Assume that $2k+3 \notin \N$ and that $M$ is an irreducible \hw\ $\szamk$-module $M$ whose \hwv\ $u$ has $T_0$-eigenvalue $\Delta$ and $W_0$-eigenvalue $w$.  Then:
	\begin{itemize}
		\item $\rmod{M}{\lambda}$ is an indecomposable $\sbpk$--module.
		\item $\rmod{M}{\lambda}$ is irreducible if and only if the polynomial $p^{\Delta,w}_k$, defined in \eqref{eq:defpoly}, has no roots in the coset $\lambda + \Z$.
		\item $\rmod{M}{\lambda}$ has no \hwvs\ and its \chwvs\ are precisely the $u \otimes e^{-j+\mu c}$ for which $\mu \in \lambda + \Z$ satisfies $p^{\Delta,w}_k(\mu) = 0$.
	\end{itemize}
\end{theorem}

\begin{remark} \label{rem:examples}
	\leavevmode
	\begin{itemize}
		\item When $k=-\frac{9}{4}$, $\szam{-9/4} = \C\vac$ and so \cref{thm:simpleembed} gives one family of \rhw\ $\sbp{-9/4}$-modules.  This family was first constructed in \cite[Thm.~7.2]{AK}.
		\item When $k=-\frac{5}{3}$, we also have $\szam{-5/3} = \C\vac$, hence one family of \rhw\ $\sbp{-5/3}$-modules.  This family may be constructed by noting that $\sbp{-5/3}$ is a $\Z_3$-orbifold of the rank-$1$ bosonic ghost system \cite[Prop.~5.9]{AK} and that this ghost system admits a family of \rhwms\ \cite{RidBos14}.
		\item When $p=2$, so $k \in \{-\frac{3}{2}, -\frac{1}{2}, \frac{1}{2}, \dots\}$, $\sbpk$ is rational \cite{ArBP}.  It therefore has no such families of \rhwms.
		\item When $k=-1$, $\sbp{-1}$ is isomorphic to the rank-$1$ Heisenberg \voa\ \cite{AdaCon18}.  It therefore also has no such families of \rhwms.
		\item For $k \in \N$, the relations $(G^{\pm}  _{-1})^{k+2} \vac = 0$ in $\sbpk$ (\cite{AK} and \cref{lem:easycalc} above) imply that the Zhu algebra $\zhu{\sbpk}$ has only finite-dimensional irreducible modules, more precisely modules of dimension at most $k+2$.  As there are no infinite-dimensional irreducible $\zhu{\sbpk}$-modules, $\sbpk$ likewise has no such families of \rhwms.
	\end{itemize}
\end{remark}

\begin{remark}
	In the case of the affine vertex algebras $\uaff{k}{\SLA{sl}{2}}$ and $\saff{k}{\SLA{sl}{2}}$, the irreducibility of the \rhwms\ was discussed in \cite{AdaRea19} and \cite{KR-2019}, using different techniques. As a consequence of our results (with some minor modifications), we can now give a new proof of the irreducibility of these modules.

	In particular, \cite{AdaRea19} showed that all the \rhw\ $\saff{k}{\SLA{sl}{2}}$-modules could be realised in the form $\mathcal E^ {\lambda} _{r,s} = M_{r,s} \otimes \Pi_{-1} (\lambda)$, where
	\begin{equation}
		k+2 = \frac{p'}{p}, \quad \text{for some coprime}\ p,p' \in \Z_{\ge2},
	\end{equation}
	and, for $r=1,\dots,p-1$ and $s=1,\dots,p'-1$, $M_{r,s}$ is the irreducible \hw\ Virasoro module of central charge and conformal weight
	\begin{equation}
		c_{p,p'} = 1 - \frac{6 (p-p')^2}{pp'} \quad \text{and} \quad h_{r,s} = \frac{(p'r-ps)^2-(p-p')^2}{4pp'},
	\end{equation}
	respectively.  Completely analogous arguments to those resulting in \cref{thm:rmodirred,thm:rmodirred-simple} then prove that $\mathcal E^ {\lambda} _{r,s}$  is irreducible if and only if $\lambda \notin  \lambda^{\pm} _{r,s } + {\Z}$, where $  \lambda^{\pm} _{r,s }$ is as in \cite[Sec.~7]{AdaRea19}.
\end{remark}

\section{Critical-level results} \label{sec:critlev}

A critical level definition of the  \bp\ algebra was investigated in \cite{Ar-crit,GK}.
\begin{definition} \label{def:critbp}
	At the critical level $k=-3$, the \bp\ vertex algebra $\ubp{-3}$ is the universal vertex algebra generated by fields $S$, $J$, $G^+$ and $G^-$ subject to the following \opes:
	\begin{equation} \label{ope:crit}
		\begin{gathered}
			J(z)J(w)\sim -\frac{1}{(z-w)^2}, \qquad
			J(z)G^{\pm}(w)\sim\pm \frac{G^{\pm}(w)}{z-w}, \qquad
			G^{\pm}(z)G^{\pm}(w)\sim0, \\
			S(z)G^{\pm}(w)\sim 0, \qquad
			S(z)J(w)\sim 0, \qquad
			S(z)S(w)\sim 0, \\
			\begin{aligned}
				G^+(z)G^-(w)\sim &\frac{6}{(z-w)^3}-\frac{6J(w)}{(z-w)^2} +\frac{3:J(w)J(w): - 3 \pd J(w) - S(w)}{z-w}.
			\end{aligned}
		\end{gathered}
	\end{equation}
	%We shall likewise denote the simple quotient of $\ubp{-3}$ by $\sbp{-3}$.
\end{definition}

\begin{remark}
	We can formally obtain the definition of $\ubp{-3}$ given above by substituting $S=(k+3)L$ into \cref{def:bp} and then setting $k=-3$.
\end{remark}

Denote the centre of $\usl{-3}$ by $\centusl$.  It is a commutative vertex algebra generated by two fields \cite{FeiAff92}
\begin{equation}
	S^2(z)  = \sum_{n\in\Z} S^2_n z^{-n-2} \qquad \text{and} \qquad S^3(z) = \sum_{n\in\Z} S^3_n z^{-n-3}.
\end{equation}
Define the  operator $d \in \End (\centusl )$ by
\begin{equation}
	[ d, S^m_n] = -n S^m_n, \quad n \in {\Z},\ m = 2,3.
\end{equation}
By setting $d \vac = 0$, this gives $\centusl$ the structure of a $\Z_{\ge 0} $-graded vertex algebra.

Direct calculation now gives the following critical-level version of \cref{thm:realisation}.  The injectivity follows in exactly the same way as in \cref{sec:injective}.
\begin{theorem} \label{thm:critrealisation}
	At the critical level $k=-3$, there is an \emph{injective} vertex algebra homomorphism $\phi^{-3} \colon \ubp{-3} \to \centusl \otimes \lvoa$, uniquely determined by
	\begin{equation} \label{eq:critrealisation}
		\begin{gathered}
			G^+ \mapsto \vac \otimes e^c, \qquad J \mapsto \vac \otimes j, \qquad S \mapsto S^2 \otimes \vac, \\
			\begin{aligned}
				G^- &\mapsto \left(S^3 - \frac{1}{2} \pd S^2\right) \otimes e^{-c} + S^2 \otimes i_{(-1)} e^{-c} \\
				&\mspace{40mu} - \vac \otimes \left(i_{(-1)}^3 - 3 i_{(-2)} i_{(-1)} + 2 i_{(-3)}\right) e^{-c}.
			\end{aligned}
		\end{gathered}
	\end{equation}
\end{theorem}
\begin{remark}
	It is proved in \cite{FeiAff92}  that $\centusl$ is isomorphic to the critical-level principal W-algebra $\uwalg{-3}{\SLA{sl}{3}}{f_{\text{pr}}} = \uzam{-3}$.
\end{remark}

Since $ \centusl $ is a commutative vertex algebra generated by $S^2(z)$ and $S^3(z)$, its irreducible modules are $1$-dimensional and parametrised by $\chi_2, \chi_3 \in {\C}((z))$ such that
\begin{equation}
	\chi_m (z) = \sum_{n \in {\Z}} \chi_m (n) z^{-n-m}, \quad m=2,3,
\end{equation}
and $S^m_n$ acts as multiplication by $\chi_m (n)$ on the irreducible module.  We shall therefore denote the irreducible modules by $L_{\chi_2,\chi_3}$.

Consider the ``relaxed'' $\ubp{-3}$-module
\begin{equation}
	R_{\chi_2, \chi_3}(\lambda) = L_{\chi_2, \chi_3} \otimes \lmod{-1}{\lambda}
\end{equation}
(actually, this module is of Wakimoto type \cite{FreVer01}).  The question of when this module is irreducible may be treated using methods from \cite{Ad-2007} and we hope to study this in forthcoming publications.

Here, we consider the case in which $R_{\chi_2, \chi_3} (\lambda)$ is $\Z_{\ge 0}$-gradable.  Since \cref{thm:critrealisation} gives $\ubp{-3}$ the grading defined by $d + t_0$, where $t$ is the conformal vector of $\lvoa$ (see \cref{sec:lvoa}), $R_{\chi_2, \chi_3} (\lambda)$ will be $\Z_{\ge 0}$-gradable if and only if $L_{\chi_2,\chi_3}$ is gradable by $d$.  This, in turn, requires that the $S^m_n$, $m=2,3$ and $n\in\Z$, act trivially unless $n=0$.  We therefore conclude that $R_{\chi_2, \chi_3} (\lambda)$ is $\Z_{\ge 0}$-gradable if and only if
\begin{equation} \label{characters-1}
	\chi_2 (z) = \frac{\Delta}{z^2} \quad \text{and} \quad \chi_3(z) = \frac{w}{z^3},
\end{equation}
for some $ \Delta, w \in {\C}$.  Moreover, $\tp{R_{\chi_2, \chi_3} (\lambda)} = L_{\chi_2,\chi_3} \otimes \tp{\lmod{-1}{\lambda}}$ will then be an $\zhu{\ubp{-3}}$-module.

\begin{theorem}
	Assume that $\chi_2(z)$ and $\chi_3(z)$ are given by \eqref{characters-1} such that
	\begin{equation}
		g^{\Delta,w}(x) = w + \Delta +  (\Delta-2)  x -  3x^2-x^3
	\end{equation}
	has no roots in the coset $\lambda + {\Z}$. Then, $ R_{\chi_2, \chi_3} (\lambda)$ is an irreducible  $\ubp{-3}$-module.
\end{theorem}
\begin{proof}
	As in the proof of \cref{thm:rmodirred}, $ R_{\chi_2, \chi_3} (\lambda)$ is an irreducible $\ubp{-3}$-module if and only if $\tp{R_{\chi_2, \chi_3} (\lambda)}$ is an irreducible $\zhu{\ubp{-3}}$-module.  Moreover, this holds if and only if $G^+_0$ and $G^-_0$ act bijectively on $\tp{R_{\chi_2, \chi_3} (\lambda)}$.
	Here, we set $G^+(z)=\sum_n G^+_nz^{-n-1}$ and $G^-(z)=\sum_n G^-_nz^{-n-2}$.
	For $u \in  L_{\chi_2, \chi_3}$, we have
	\begin{equation}
		\begin{aligned}
			G^+_0 (u \otimes e^{-j+(\lambda+n)c}) &= u \otimes e^c_0 e^{-j+(\lambda+n)c} = u \otimes e^{-j+(\lambda+n+1)c} \\ \text{and} \quad
			G^-_0 (u \otimes e^{-j+(\lambda+n)c}) &=  g^{\Delta,w}(\lambda + n -1) u \otimes e^{-j+(\lambda+n-1)c},
		\end{aligned}
	\end{equation}
	so $G^+_0$ always acts bijectively and $G^-_0$ acts bijectively if and only if $g^{\Delta,w}(x)$ has no roots in $\lambda+\Z$.
	The proof follows.
\end{proof}

\flushleft
%\bibliography{bp}
%\bibliographystyle{plain} %DR alphabetical order just for KK :)
\providecommand{\opp}[2]{\textsf{arXiv:\mbox{#2}/#1}}
\providecommand{\pp}[2]{\textsf{arXiv:#1 [\mbox{#2}]}}

\end{document}